\documentclass[reqno,11pt]{amsart}
 
\usepackage{amsmath}
\usepackage{amsthm}  
\usepackage{verbatim}
\usepackage{enumerate} 
\usepackage{mathtools}  
\usepackage{amssymb}
\usepackage{mathrsfs}
\usepackage[top=1.5in, bottom=1.5in, left=0.7in, right=0.7in]{geometry}  
\usepackage[colorlinks]{hyperref}
\hypersetup{
	colorlinks,
	citecolor=blue,
	linkcolor=red
}   
\numberwithin{equation}{section}
\usepackage[all]{xy}
\usepackage{bm}
\usepackage{tikz}
\usetikzlibrary{decorations.pathreplacing}
\usepackage{xcolor}
\usepackage{float}

\newtheorem{theorem}{\textbf{Theorem}}[section]
\newtheorem{theorem*}{\textbf{Theorem}}
\newtheorem{definition}[theorem]{\textbf{Definition}}
\newtheorem{proposition}[theorem]{\textbf{Proposition}}

\newtheorem{corollary}[theorem]{\textbf{Corollary}}
\newtheorem{remark}[theorem]{\textbf{Remark}}
\newtheorem{assumption}[theorem]{\textbf{Assumption}}
\newtheorem{example}[theorem]{\textbf{Example}}

\newtheorem{definition/proposition}[theorem]{\textbf{Definition/Proposition}}

\def\N{{\mathbb N}}
\def\R{\mathbb{R}}
\def\Z{{\mathbb Z}}

\def\C{{\mathbb C}}

\newcommand{\CP}{\mathbb{C}\mathbb{P}}

\def\a{\alpha}

\def\cA{{\mathcal A}}
\def\cB{{\mathcal B}}
\def\cC{{\mathcal C}}
\def\cD{{\mathcal D}}
\def\cE{{\mathcal E}}
\def\cF{{\mathcal F}}

\def\cH{{\mathcal H}}

\def\cJ{{\mathcal J}}

\def\cM{{\mathcal M}}

\def\cP{{\mathcal P}}

\def\cS{{\mathcal S}}

\def\cU{{\mathcal U}}

\def\fH{{\mathfrak H}}

\def\fP{{\mathfrak P}}

\def\rD{{\rm D}}

\def\rd{{\rm d}}

\DeclareMathOperator{\Ima}{im}
\DeclareMathOperator{\ind}{ind}

\DeclareMathOperator{\Id}{id}

\DeclareMathOperator{\mor}{Mor}
\DeclareMathOperator{\obj}{Obj}

\title{On the cohomology ring of symplectic fillings}
\author{Zhengyi Zhou}

\begin{document}
	\maketitle
\begin{abstract}
We consider symplectic cohomology twisted by sphere bundles, which can be viewed as an analogue of symplectic cohomology with local systems. Using the associated Gysin exact sequence, we prove the uniqueness of part of the ring structure on cohomology of fillings for those asymptotically dynamically convex manifolds with vanishing property considered in \cite{zhou2017vanishing,zhou2019symplectic}. In particular, for any simply connected $4n+1$ dimensional flexibly fillable contact manifold $Y$, we show that the real cohomology $H^*(W)$ is unique as a ring for any Liouville filling $W$ of $Y$ as long as $c_1(W)=0$. Uniqueness of real homotopy type of Liouville fillings is also obtained for a class of flexibly fillable contact manifolds.
\end{abstract}
\tableofcontents
\section{Introduction}
It is conjectured that Liouville fillings of certain contact manifolds are unique. The first result along this line is by Gromov \cite{gromov1985pseudo} that Liouville fillings of the standard contact $3$-sphere are unique. The dimension $4$ case is special because of the intersection theory of $J$-holomorphic curves. For higher dimensional cases, only weaker assertions can be made so far. Eliashberg-Floer-McDuff \cite{mcduff1991symplectic} proved that any symplectically aspherical  filling of the standard contact sphere of dimension $\ge 5$ is diffeomorphic to a ball. Oancea-Viterbo \cite{oancea2012topology} showed that $H_*(Y;\Z)\to H_*(W;\Z)$ is surjective for a simply-connected subcritically fillable contact manifold $Y$ and any symplectically aspherical $W$. Barth-Geiges-Zehmisch \cite{barth2016diffeomorphism} generalized the Eliashberg-Floer-McDuff theorem to the subcritically fillable case assuming $Y$ is simply connected and of dimension $\ge 5$. Roughly speaking, the method used to obtain above results is finding a ``homological foliation", which is hinted at by the splitting result of Cieliebak \cite[Theorem 14.16]{MR3012475} for subcritical domains. 

On the other hand, contact manifolds considered above are asymptotically dynamically convex (ADC) in the sense of Lazarev \cite{lazarev2016contact}, which is a much larger class of contact manifolds. Those contact manifolds admit only the trivial (contact DGA) augmentation due to degree reasons, hence one may expect the filling is rigid in some sense. In \cite{zhou2019symplectic2,zhou2017vanishing,zhou2019symplectic}, we studied fillings of such manifolds using Floer theories. Roughly speaking, a contact manifold is ADC iff the symplectic field theory (SFT) grading is positive (asymptotically).  By the neck-stretching argument, such condition is sufficient to prove invariance of many structures in \cite{zhou2019symplectic2,zhou2019symplectic}. However, in many cases, the SFT gradings are greater than some positive integer $k$, which provides more room in the neck-stretching argument. The goal of this paper is trying to make use of such extra room and getting more information. In particular, we will study the ring structure of symplectic fillings. Our main theorem is the following. We call a Liouville filling $W$ of $Y$ \emph{topologically simple} iff $c_1(W)=0$ and $\pi_1(Y)\to \pi_1(W)$ is injective. Throughout this paper, the default coefficient is $\R$.
\begin{theorem}\label{thm:main}
Let $Y$ be a $k$-ADC contact manifold (Definition \ref{def:kADC}) with a topologically simple Liouville filling $W_1$ and $SH^*(W_1)=0$. Then for any topologically simple Liouville filling $W_2$, there is a linear isomorphism $\phi: H^*(W_1) \to H^*(W_2)$ preserving grading, such that $\phi(\alpha \wedge \beta) = \phi(\alpha)\wedge \phi(\beta)$ for all $\alpha \in H^{2m}(W_1)$ with $2m\le k+1$.
\end{theorem}

The main example where Theorem \ref{thm:main} can be applied is flexibly fillable contact manifold $Y^{2n-1}$, which is $(n-3)$-ADC by Lazarev \cite{lazarev2016contact}. In particular, combining with \cite[Corollary B]{zhou2019symplectic}, we have the following.

\begin{corollary}\label{cor:ring}
	Let $Y$ be a simply connected $4n+1$ dimensional flexibly fillable contact manifold with $c_1(Y)=0$, then the real cohomology ring of Liouville fillings of $Y$ with vanishing first Chern class is unique.
\end{corollary}

\begin{remark}
	By \cite[Corollary B]{zhou2019symplectic}, manifolds considered in Theorem \ref{thm:main} has the property that $H^*(W;\Z)\to H^*(Y;\Z)$ is independent of topologically simple Liouville fillings. Therefore on the degree region where the restriction map is injective, we can infer the ring structure of the filling from the boundary. However, such method can not yield information when the product lands in a degree of where the restriction is not injective, i.e.\ the middle degree for flexibly fillable manifolds. The method used in this paper asserts the uniqueness of product structure in those ambiguous regions.
\end{remark}

There are also non-Weinstein examples that Theorem \ref{thm:main} can be applied, see \S \ref{s2}. If $Y$ is subcritically fillable, then the $\pi_1$-injective condition can be dropped because Reeb orbits can be assumed to be contractible \cite{lazarev2016contact}. However this case is covered by both \cite[Corollary B]{zhou2019symplectic} and \cite[Theorem 1.2]{barth2016diffeomorphism} along with the universal coefficient theorem. 

In some cases, the knowledge of the cohomology ring is enough to determine the real homotopy type. In particular, we have the following corollary. 
\begin{corollary}\label{cor:real}
	Let $M$ be the product of $\CP^n, \mathbb{HP}^n, S^{2n}$ and at most one copy of $S^{2n+1}$ for $n\ge 1$, let $Y$ denote the contact boundary of the flexible cotangent bundle of $M$. Then the real homotopy type of Liouville fillings of $Y$ is unique, as long as the Liouville filling has vanishing first Chern class.
\end{corollary}

The method used in this paper is very different from the method used in \cite{barth2016diffeomorphism,mcduff1990structure,oancea2012topology}, where they studied the moduli spaces of $J$-holomorphic curves in a partial compactification of $W$. The essential property needed for the partial compactification is that $W$ splits to $V\times \C$ with $V$ Weinstein. However, for many flexible critical Weinstein domains, such splitting does not exist even in the topology category, e.g.\ flexible cotangent bundles $T^*S^{2n}$ can not be written as a complex line bundle over a manifold for $n>1$\footnote{Assume otherwise, $T^*S^{2n}$ can be written as a complex line bundle over some manifold $V$ with boundary, since $H^2(V;\Z)=0$ when $n>1$, the complex line bundle is necessarily trivial. Therefore we have $T^*S^{2n}=V\times \mathbb{D}$, where $\mathbb{D}$ is the unit disk in $\C$. On the other hand, $H^*(T^*S^{2n};\Z)\to H^*(\partial T^*S^{2n};\Z)$ is not injective (in degree $2n$) but $H^*(V\times \mathbb{D};\Z)\to H^*(\partial (V\times \mathbb{D});\Z)$ is always injective, hence we arrive at a contradiction.}. Our method is based on symplectic cohomology and uses the index property of the contact boundary, hence we need to assume $c_1=0$. The strategy of the proof is representing the cup product as a multiplication with an Euler class of a sphere bundle. Therefore we consider symplectic cohomology twisted by sphere bundles, which leads to Gysin exact sequences.  The Gysin exact sequence associated to a $k$-sphere bundle uses moduli spaces of dimension up to $k$. We show that the Gysin exact sequence for a $k$-sphere bundle on the positive symplectic cohomology is independent of the filling by a neck-stretching argument, if the boundary is $k$-ADC. Then we can relate it to the regular Gysin sequence of the filling by the vanishing result in \cite{zhou2017vanishing}. 

\begin{remark}
The reason of restricting to real coefficient is twofold. Firstly, it is not true that every class in $H^{2k}(M;\Z)$ can be represented as the Euler class of an oriented vector bundle \cite{walschap2002euler} unless multiplied by a large integer \cite{guijarro2002bundles}, which only depends on the degree and dimension. Secondly, the Gysin exact sequence is derived from the Morse-Bott framework developed in \cite{MB}, which is defined over $\R$. In the case considered in this paper, one can get Gysin exact sequences in $\Z$-coefficient, since our moduli spaces do not have isotropy or weight. For example, one can generalize the Morse-Bott construction in \cite{hutchings2017axiomatic} to sphere bundles to prove a $\Z$-coefficient Gysin exact sequence.
\end{remark}
\begin{remark}
	In this paper, by symplectic cohomology we mean the symplectic cohomology generated by \emph{contractible} orbits. The role of topologically simplicity of the filling is to guarantee that the symplectic cohomology of the filling is canonically graded by $\Z$ using any trivialization of $\det \xi$ on $Y$. From the SFT perspective, it is related to that the augmentation from the filling is (canonically) graded by $\Z$. Since the ADC condition only asserts unique contact DGA augmentation with a $\Z$ grading, we can only hope for uniqueness for topologically simple fillings using the ADC condition, also see \cite[Remark 3.6]{zhou2019symplectic}.
\end{remark}

\subsection*{Organization of the paper}
\S \ref{s2} reviews the contact geometry background and provides a list of examples where Theorem \ref{thm:main} applies.  In \S \ref{s3}, we define the symplectic cohomology of sphere bundles and prove the independence result when the boundary is $k$-ADC. We finish the proof of Theorem \ref{thm:main}, its corollaries and applications in \S \ref{s4}.

\subsection*{Acknowledgements}
The author is supported  by  the National Science Foundation under Grant No. DMS-1638352.  It is a great pleasure to acknowledge the Institute for Advanced Study for its warm hospitality.  The author would like to thank the referee for many suggestions that improved this paper. This paper is dedicated to the memory of Chenxue.
\section{Asymptotically dynamically convex manifolds}\label{s2}
Let $\alpha$ be a contact form of $(Y^{2n-1},\xi)$ and $D > 0$, we use $\cP^{<D}(Y,\alpha)$ to denote the set of contractible Reeb orbits of $\alpha$ with period smaller than $D$. Let $\alpha_1,\alpha_2$ be two contact forms of $(Y,\xi)$, we write $\alpha_1 \ge \alpha_2$ if $\alpha_1 = f\alpha_2$ for $f\ge 1$. For a non-degenerate Reeb orbit $\gamma$, the degree is defined to be $\mu_{CZ}(\gamma)+n-3$, which is canonically defined in $\Z$ if $c_1(\xi)=0$ and $\gamma$ is contractible.
\begin{definition}[{\cite[Definition 3.6]{lazarev2016contact}}]\label{def:ADC}
	A contact manifold $(Y, \xi)$ with $c_1(\xi)=0$ is asymptotically dynamically convex (ADC) if there exist a non-increasing sequence of  contact forms $\a_1 \ge \a_2 \ge \a_3\ge \ldots$ for $\xi$ and increasing
	positive numbers $D_1 < D_2 < D_3 < \ldots$ going to infinity, such that all elements of $\cP^{<D_i} (Y, \a_i)$ are non-degenerate and have positive degree.
\end{definition}

One important consequence of ADC is that the positive symplectic cohomology is independent of the filling $W$ whenever $c_1(W)=0$ and $\pi_1(Y)\to \pi_1(W)$ is injective \cite[Proposition 3.8]{lazarev2016contact}.  Moreover, many Floer theoretic properties of the filling are independent of fillings \cite{zhou2019symplectic2,zhou2019symplectic}. Roughly speaking, ADC guarantees the $0$-dimensional moduli spaces used in the definition of the positive symplectic cohomology are completely contained in the cylindrical end of the completion $\widehat{W}$, hence are independent of the filling. In this paper, we consider sphere bundles over (positive) symplectic cohomology. The information for $k$-sphere bundles is encoded in moduli spaces with dimension up to $k$. In particular, the associated Gysin exact sequence depends on moduli spaces with dimension up to $k$. Therefore we need more positivity in the degree of Reeb orbits, the following finer dynamical convexity is needed.
\begin{definition} \label{def:kADC}
	A contact manifold $(Y, \xi)$ with $c_1(\xi)=0$ is $k$-ADC if there exist a non-increasing sequence of contact forms $\a_1 \ge \a_2 \ge \a_3 \ge \ldots$ for $\xi$ and increasing positive numbers $D_1 < D_2 < D_3 < \ldots$ going to infinity, such that all elements of $\cP^{<D_i} (Y, \a_i)$ are non-degenerate and have degree $>k$. 
	
	Similarly,  we say a Liouville domain $(W,\lambda)$ with $c_1(W)=0$ is $k$-ADC, iff there exist positive functions $f_1\ge f_2\ge \ldots$ and positive numbers $D_1<D_2<D_3<\ldots$ going to infinity, such that all contractible (in $W$) orbits of $(\partial W, f_i\lambda)$ of periods up to $D_i$ are non-degenerate and have degree $>k$.
\end{definition}
In particular, $(k+1)$-ADC implies $k$-ADC and $0$-ADC is the usual ADC condition in \cite{lazarev2016contact}. The basic example of $k$-ADC manifold is the standard contact sphere $S^{2n-1}$, which is $(2n-3)$-ADC. From this basic example, the following propositions yield many $k$-ADC manifolds.

\begin{proposition}[{\cite[Theorem 3.15, 3.17, 3.18]{lazarev2016contact}}]\label{prop:handle}
	Let $Y$ be a $(2n-3-k)$-ADC contact manifold, then the attachment of a index $k\ne 2$ subcritical/flexible handle to $Y^{2n-1}$ preserves the $(2n-3-k)$-ADC property. When $k=2$, the same holds if the conditions in \cite[Theorem 3.17]{lazarev2016contact} are met.
\end{proposition}

Let $V$ be a manifold with boundary, we define the Morse dimension $\dim_MV$ to the minimal value of the maximal index of an exhausting Morse function on $V$.

\begin{proposition}[{\cite[Theorem 6.3]{zhou2019symplectic}}]\label{prop:disk}
	Let $V^{2n}$ be a Liouville domain with $c_1(V)=0$. Then $\partial(V\times \C)$ is $(2n-1-\dim_M V)$-ADC.
\end{proposition}

\begin{proposition}[{\cite[Theorem 6.19]{zhou2019symplectic}}]\label{prop:product}
	Let $V,W$ be $p,q$-ADC domains respectively. Then $\partial(V\times W)$ is $\min\{p+q+4,p+\dim W-\dim_MW, q+ \dim V- \dim_M V\}$-ADC.
\end{proposition}

\begin{example}\label{ex:ex}
Using the above three propositions, we have the following classes of contact manifolds that Theorem \ref{thm:main} can be applied to.
	\begin{enumerate}
		\item By Proposition \ref{prop:handle}, for $n\ge 3$, any $2n-1$ dimensional flexibly fillable contact manifold $Y$ with $c_1(Y)=0$ is $(n-3)$-ADC.
		\item Let $V$ be the $2n$ dimensional Liouville but not Weinstein domain constructed in \cite{massot2013weak}, then $\partial (V\times \C^k)$ is $(2k-2)$-ADC by Proposition \ref{prop:disk}.
		\item Products of any $k$-ADC domain for $k>0$ with an example from the above two classes are $m$-ADC for a suitable $m>0$ by Proposition \ref{prop:product}. We can also attach flexible handle afterwards. 
	\end{enumerate}
\end{example}
In general, there are many more $k$-ADC contact manifolds of interests, e.g.\ cotangent bundles, links of terminal singularities. For certain cotangent bundles, symplectic cohomology is zero with an appropriate local system \cite{albers2017local}. In general, symplectic cohomology in these cases is not zero, hence they are beyond the scope of this paper. 
\section{Symplectic cohomology and fiber bundles}\label{s3}
In this section, we review some basic properties of symplectic cohomology associated to a Liouville domain \cite{cieliebak2015symplectic,ritter2013topological,seidel2008biased}. Then we introduce the symplectic cohomology of sphere bundles and the associated Gysin exact sequences using the abstract Morse-Bott framework developed in \cite{MB}.
\subsection{Symplectic cohomology}\label{sub:sc}
\subsubsection{Floer cochain complexes}
To a Liouville filling $(W,\lambda)$ of the contact manifold $(Y,\xi)$, one can associate the completion $(\widehat{W}, \rd \widehat{\lambda})=(W\cup_{Y}[1,\infty)_r \times Y, d\widehat{\lambda})$, where $\widehat{\lambda}=\lambda$ on $W$ and $\widehat{\lambda}=r(\lambda|_{Y})$ on $[1,\infty)_r\times Y$. Let $H:S^1\times \widehat{W}\to \R$ be a Hamiltonian, our convention for the Hamiltonian vector field is 
$$\omega(\cdot, X_{H}) = \rd H.$$
Then symplectic cohomology is defined as the ``Morse cohomology" of the symplectic action functional
\begin{equation}\label{eqn:action}
\cA_H(x):=  -\int x^*\widehat{\lambda}+\int H\circ x(t) \rd t,
\end{equation}
for a Hamiltonian $H=r^2$ for $r\gg0$ \cite{ritter2013topological,seidel2008biased}. Equivalently, one may define symplectic cohomology as the direct limit of Hamiltonian-Floer cohomology of $H=Dr$ for $r\gg 0$ as $D$ goes to infinity. For simplicity of the construction, we will use the former construction and a special class of Hamiltonian in this paper. Let $\alpha$ be a non-degenerate contact form of the contact manifold $(Y,\xi)$ and $R_\alpha$ its associated Reeb vector field, then we define 
$$\cS(Y,\alpha):= \left\{\int_{\gamma}\alpha\left|\gamma \text{ is periodic oribit of } R_\alpha\right.\right\}.$$
Following \cite{bourgeois2009symplectic}, we can choose a smooth family of time-dependent Hamiltonian $H_R$ for $R\in [0,1]$ as a careful perturbation of an autonomous Hamiltonian, such that the following holds.
\begin{enumerate}
	\item $H_R|_W$ is time independent $C^2$-small Morse for $R\ne 0$ and $H_0|_W=0$.
	\item There exist sequences of non-empty open intervals $(a_0,b_0),(a_1,b_1),\ldots$ with $a_i,b_i$ converging to infinity and $a_0=1$, such that $H_R|_{Y\times(a_i,b_i)}=f_{i,R}(r)$ with $f_{i,R}''>0$ and $f_{i,R}'\notin \cS(Y,\alpha)$ and $\lim_{i} \min f'_{i,R}=\infty$.
	\item $H_R$ outside $r=b_0$ does not depend on $R$.
	\item For $R\ne 0$, the periodic orbits of $X_{H_R}$ are non-degenerate, and are either critical points of $H_R|_{W}$ or non-constant orbits in $\partial W \times [b_i,a_{i+1}]$.
	\item\label{h4} There exist $0<D_0<D_1<\ldots \to \infty$, such that all periodic orbits of $X_{H_R}$ of action $>-D_i$ is contained in $W^i:=\{r<a_i\}$.
	\item $\partial_R H_R\le 0$. 
\end{enumerate}
We use $\cC(H_R)$ to denote the set of critical points of $H_R$ on $W$ and $\cP^*(H)$ to denote the set of non-constant \emph{contractible} orbits of $X_{H_R}$ outside $W$, which does not depend on $R$.

\begin{figure}[H]
	\begin{tikzpicture}
	\draw[->] (-3,0) -- (7,0) node[right] {$r$};
	\draw (0,0) .. controls (2,0) and (3,3) .. (4,5);
	\draw (0,-0.2) .. controls (2,0) and (3,3) .. (4,5);
	\draw (-3,-0.3) .. controls (-2,0) and (-1,-0.5) .. (0,-0.2);
	\draw (0,-0.4) .. controls (2,0) and (3,3) .. (4,5);
	\draw (-3,-0.6) .. controls (-2,0) and (-1,-0.7) .. (0,-0.4);
	\node at (-3.4,0.2) {$H_0$};
	\node at (-3.4,-0.3) {$H_{\frac{1}{2}}$};
	\node at (-3.4,-0.8) {$H_1$};
	\draw[dashed] (3,0) to (3,2.9);
	\node at (3,-0.3) {$b_0$};
	\node at (0,-0.3) {$a_0$};
	\node at (4,-0.3) {$a_1$};
	\node at (6,-0.3) {$b_1$};
	\fill (0,0) circle[radius=2pt];
	\fill (3,0) circle[radius=2pt];
	\fill (4,0) circle[radius=2pt];
	\fill (6,0) circle[radius=2pt];
	\draw [decorate,decoration={brace,amplitude=10pt},]
	(-3,0) -- (0,0) node [black,midway,yshift=0.6cm]
	{ $W$};
	\draw [decorate,decoration={brace,amplitude=10pt},]
	(0,0) -- (7,0) node [black,midway,yshift=0.6cm]
	{$\partial W\times (1,\infty)$};
	\end{tikzpicture}
	\caption{Graphs of $H_R$}
\end{figure}

\begin{remark}
	A few remarks regarding our choice of Hamiltonian are in order.
	\begin{enumerate}
		\item In this paper, we do not define symplectic cohomology of sphere bundles as an invariant, but rather using one model to infer topological information. Therefore we choose to work with one specific Hamiltonian.
		\item The requirement of $H_R$ on interval $(a_i,b_i)$ is for purpose of the integrated maximum principle \cite{abouzaid2010open,cieliebak2015symplectic}. In particular, with an admissible almost complex structure in Definition \ref{def:admissble}, any Floer cylinder asymptotic to orbits in $W^i$ will be completely contained in $W^i$.
		\item Ideally, we would like to work with $H_0$ where the neck-stretching argument will be cleaner. $H_0$ can be viewed as a ``Morse-Bott" situation, which is used in \cite{zhou2019symplectic}.  In this paper, we will use non-degenerate Hamiltonian $H_R$ for $R>0$ to approximate $H_0$,  because the relevant polyfolds are easier to construct and partially exist in literature, see Remark \ref{rmk:polyfold}. 
		\item The requirement of $\partial_R H_R\le 0$ is to have the continuation map from $H_{R_+}$ to $H_{R_-}$ respecting the action filtration for $R_+>R_-$. The independence of $H_R$ outside $r=b_0$ simplifies the continuation map for the positive symplectic cohomology to the identity map for different $R$.
	\end{enumerate}
\end{remark}

For admissible Hamiltonian $H_R$, there are infinitely many periodic orbits and they are not bounded in the $r$-coordinate. To guarantee the compactness of moduli spaces, we need to use the following almost complex structure so that the integrated maximum principle \cite{abouzaid2010open} can be applied. 
\begin{definition}\label{def:admissble}
	An $S^1$-dependent almost complex structure $J_t$ is admissible if the following holds.
\begin{enumerate} 
	\item $J_t$ is compatible with $\rd \widehat{\lambda}$ on $\widehat{W}$.
	\item $J_t$ is cylindrical convex on $\partial W \times (a_i,b_i)$, i.e.\ $\widehat{\lambda}\circ J_t =\rd r$.
	\item $J_t$ is only required to be $S^1$-independent on $W$. We will often abbreviate $J_t$ by $J$ for simplicity.
\end{enumerate}
The set of admissible almost complex structure is denoted by $\cJ(W)$.
\end{definition}

Let $x,y\in \cC(H_R)\cup \cP^*(H)$ for $R>0$ and $J$ an admissible almost complex structure, we use $\cM_{x,y,H_R}$ to denote the compactified  moduli space of solutions to the following equation modulo the $\R$-translation
\begin{equation}\label{eqn:floer}
\partial_su+J(\partial_t u-X_{H_R}) = 0, \quad \lim_{s\to \infty} u = x, \quad \lim_{s\to -\infty} u = y.
\end{equation}
We will suppress $H_R$ when there is no confusion. Then we have the following regularity result.
\begin{proposition}\label{prop:moduli}
	For any $R>0$, there exists a subset $\cJ^R(W)\subset \cJ(W)$ of second Baire category such that the following holds.
	\begin{enumerate}
		\item $\cM_{x,y}$ is a compact smooth manifold with boundary and corners for all $x,y \in \cC(H_R)\cup \cP^*(H)$.
		\item\label{bb} $\partial \cM_{x,z} = \cup_y \cM_{x,y}\times \cM_{y,z}$.
		\item $\cM_{x,z}$ can be oriented such that the induced orientation of $\partial \cM_{x,z}$ on $\cM_{x,y}\times \cM_{y,z}$ is given by the product orientation twisted by $(-1)^{\dim\cM_{x,y}}$.
		\item If $x\in \cC(H_R)$ and $y\in \cP^*(H)$, then $\cM_{x,y}=\emptyset$.
	\end{enumerate} 
\end{proposition}
This proposition is a folklore, although it is usually stated and proven for moduli spaces $\cM_{x,y}$ with virtual dimensional smaller or equal to $1$. Since $\widehat{W}$ is exact and $J$ can depend on $t\in S^1$, we have transversality for unbroken Floer trajectories. A more classical treatment to prove the first two claims is constructing compatible gluing maps for families of Floer trajectories. In the case of Lagrangian Floer theory, such construction can be found in \cite{barraud2007lagrangian}. In the case of Morse theory, a more elementary approach can be used to give the compactified moduli spaces structures of manifold with boundary and corners, see \cite{wehrheim2012smooth}. Another method is adopting the polyfold theory developed in \cite{hofer2017polyfold}. In view of this, we make the following assumption.

\begin{assumption}\label{ass}
	For any admissible almost complex structure $J$,  there exists an M-polyfold construction for the symplectic cohomology moduli spaces. More precisely, for every $x,y\in \cC(H_R)\cup \cP^*(H)$, there exists a strong tame M-polyfold bundle $\cE_{x,y}\to \cB_{x,y}$ along with an oriented proper sc-Fredholm section $s_{x,y}:\cB_{x,y}\to \cE_{x,y}$, such that the following holds.
	\begin{enumerate}
		\item $s_{x,y}^{-1}(0)=\cM_{x,y}$, where $\cM_{x,y}$ is the compact moduli space using $J$.
 		\item\label{a2} Classical transversality implies that $s_{x,y}$ is transverse and in general position.
		\item The boundary of $\cB_{x,z}$ is the union of products $\cB_{x,y}\times \cB_{y,z}$, over which the bundle and section have the same splitting.
	\end{enumerate}
\end{assumption}
\begin{remark}\label{rmk:polyfold}
	Giving a detailed proof of Assumption \ref{ass} is not the goal of this paper. Symplectic cohomology is special case of Hamiltonian-Floer cohomology, whose polyfold construction was sketched in \cite{wehrheim2012fredholm}. Alternative approach is using the full SFT polyfolds \cite{hofer2017application} as in \cite{filippenko2018polyfold}.  In those constructions, the linearization in polyfold and the linearization of Floer equation modulo $\R$-translation are the same. Then we have that classical transversality implies polyfold transversality, i.e.\ \eqref{a2} of Assumption \ref{ass} holds. We only use Assumption \ref{ass} to prove Proposition \ref{prop:moduli}, in particular, we will not use any polyfold perturbation scheme but only the existence of polyfolds. 
\end{remark}

\begin{proof}[Proof of Proposition \ref{prop:moduli}]
	To obtain the compactness of moduli spaces, in addition to including Floer breakings, we also need to rule out the possibility of a curve escaping to infinity. To this end, since we choose $J$ to be cylindrical convex on $\partial W \times (a_i,b_i)$ where $H_R=f_{i,R}(r)$, we can apply the integrated maximum principle of Abouzaid and Seidel \cite{abouzaid2010open} to any $r\in (a_i,b_i)$, also see \cite[Lemma 2.2]{cieliebak2015symplectic} for the specific version of integrated maximum principle we need here. We pick an admissible almost complex structure such that moduli spaces of unbroken Floer trajectories of any virtual dimension are cut out transversely. By Assumption \ref{ass}, we have the M-polyfolds description of compactified moduli spaces as zero sets of sc-Fredholm sections. By \eqref{a2} of Assumption \ref{ass}, those sc-Fredholm sections are cut out transversely. Then the M-polyfold implicit function theorem \cite[Theorem 3.15]{hofer2017polyfold}\footnote{\cite[Theorem 3.15]{hofer2017polyfold} is stated for sections in good position. To obtain a decomposition of the boundary in the form of \eqref{bb} of Proposition \ref{prop:moduli} for sections in general position, one also needs \cite[Theorem 4.3]{hofer2017polyfold}.} endows the compactified moduli spaces smooth structures of manifolds with boundary and corners. It is worth noting that that we only need the existence of M-polyfolds with sc-Fredholm sections without evoking any abstract perturbation scheme. In particular, the first two claims hold. The claim on orientations follows from \cite[\S 5.1.1]{MB}. If $\cM_{x,y}\ne \emptyset$, then by energy reason, we have $\cA_{H_R}(y)-\cA_{H_R}(x)\ge 0$. Then the last claim follows from property \eqref{h4} of $H_R$. 
\end{proof}

The Hamiltonian-Floer cochain complex is defined by counting the zero dimensional moduli spaces $\cM_{x,y}$. However, since we need to consider sphere bundles over the moduli spaces later, which is naturally a Morse-Bott situation, we need to introduce the Morse-Bott framework developed in \cite{MB}. To this purpose, we recall the concept of \emph{flow category}, which was first introduced in \cite{cohen1995floer}.
\begin{definition}[{\cite[Definition 2.9]{MB}}]\label{flow}
	A flow category is a small category $\cC$ with the following properties.
	\begin{enumerate}
		\item\label{F1} The objects space $Obj_\cC=\sqcup_{i\in \Z} C_i$ is a disjoint union of closed manifolds $C_i$. The morphism space $Mor_{\cC}=\cM$ is a manifold with boundary and corners. The source and target maps $s,t:\cM\to C$ are smooth. 
		\item\label{F2} Let $\cM_{i,j}$ denote $(s\times t)^{-1}(C_i\times C_j)$. Then $\cM_{i,i}=C_i$, corresponding to the identity morphisms and $s,t$ restricted to $\cM_{i,i}$ are identities. $\cM_{i,j}=\emptyset$ for $j<i$, and $\cM_{i,j}$ is a compact manifold with boundary and corners for $j>i$.
		\item\label{F3}  Let $s_{i,j},t_{i,j}$ denote $s|_{\cM_{i,j}}, t|_{\cM_{i,j}}$. For every strictly increasing sequence $i_0<i_1<\ldots<i_k$, $t_{i_0,i_1}\times s_{i_1,i_{2}}\times t_{i_1,i_2}\times \ldots \times s_{i_{k-1},i_k}:{\cM_{i_0,i_1}}\times {\cM_{i_1,i_2}}\times\ldots \times\cM_{i_{k-1},i_k}\to C_{i_1}\times C_{i_1}\times C_{i_2}\times C_{i_2}\times \ldots \times C_{i_{k-1}}\times C_{i_{k-1}}$ is transverse to the submanifold $\Delta_{i_1}\times \ldots \times \Delta_{i_{k-1}}$, where $\Delta_{i_j}$ is the diagonal in $C_{i_j}\times C_{i_j}$.  Therefore the fiber product ${\cM_{i_0,i_1}}\times_{i_1} {\cM_{i_1,i_2}}\times_{i_2}\ldots \times_{i_{k-1}}\cM_{i_{k-1},i_k}:=(t_{i_0,i_1}\times s_{i_1,i_{2}}\times t_{i_1,i_2}\times \ldots \times s_{i_{k-1},i_k})^{-1}(\Delta_{i_1}\times \Delta_{i_2}\times \ldots \times \Delta_{i_{k-1}})\subset {\cM_{i_0,i_1}}\times {\cM_{i_1,i_2}}\times\ldots \times\cM_{i_{k-1},i_k} $ is a submanifold.
		\item\label{F4} The composition $m:{\cM_{i,j}}\times_j {\cM_{j,k}}\to \cM_{i,k}$ is a smooth  map, such that $m:  \bigsqcup_{i < j < k} \cM_{i,j}\times_j \cM_{j,k}\to \partial \cM_{i,k}$ is a diffeomorphism up to zero-measure, i.e.\ $m$ is a diffeomorphism from a full measure open subset to a full measure open subset.
	\end{enumerate}
\end{definition}

In the case of Floer theory considered here, the object space is the set of critical points, the morphism space is the union of all compactified moduli spaces of Floer trajectories in addition to the identity morphisms. The source and target maps are evaluation maps at two ends and the composition is the concatenation of trajectories. The fiber product transversality is tautological, as both source and target maps map to $0$-dimensional manifolds. If we label periodic orbits  $\cC(H_R)\cup \cP^*(H)$ by integers, such that $\cA_{H_R}(x_i) \le \cA_{H_R}(x_j)$ iff $i\le j$. Then we have $\cM_{x_i,x_j} = \emptyset$ if $i > j$. Moreover, we can require that $x_i$ is a critical point of $H_R|_W$ iff $i \ge 0$ . With such labels, Proposition \ref{prop:moduli} gives three flow categories $\cC^{R,J}, \cC_0^{R,J}$ and $\cC_+^{R,J}$,
$$
\begin{array}{rclrcl}
\obj(\cC^{R,J}) & := & \{x_i\}, & \mor(\cC^{R,J}) &:= &\{ \cM_{i,j}:=\cM_{x_i,x_j}\};\\
\obj(\cC_0^{R,J}) &:=& \{x_i\}_{i\ge 0}, &  \mor(\cC_0^{R,J}) &:=& \{ \cM_{i,j}:=\cM_{x_i,x_j}\}_{i,j\ge 0};\\
\obj(\cC_+^{R,J}) &:=& \{x_i\}_{i<0}, &  \mor(\cC_+^{R,J}) &:=& \{ \cM_{i,j}:=\cM_{x_i,x_j}\}_{i,j < 0}.
\end{array}
$$
Moreover, $\cC_0^{R,J}$ is subflow category of $\cC^{R,J}$ with quotient flow category $\cC_+^{R,J}$ in the sense of \cite[Proposition 3.38]{MB}. By considering only periodic orbits of action $>-D_i$, i.e.\ those contained in $W^i$, we have two subflow categories $\cC^{R,J}_{\le i}\subset \cC^{R,J}$ and $\cC^{R,J}_{+,\le i}\subset \cC^{R,J}_{+}$. In particular, we have $\cC^{R,J}_{\le 0}=\cC^{R,J}_0$.  The orientation property of Proposition \ref{prop:moduli} implies that $\cC^{R,J},\cC^{R,J}_0,\cC^{R,J}_+$  and the truncated versions $\cC^{R,J}_{\le i},\cC^{R,J}_{+,\le i}$ are \emph{oriented flow categories} \cite[Definition 2.15]{MB}.  The main theorem of \cite{MB} is that for every oriented flow category $\cC = \{C_i,\cM_{i,j}\}$, one can associate to it a cochain complex $C^*(\cC)$ over $\R$ generated by $H^*(C_i;\R)$, whose homotopy type is well-defined. The one feature of the construction in \cite{MB} that will be used in this paper is the following.
\begin{proposition}[{\cite[Corollary 3.13]{MB}}]\label{prop:kchain}
	Let $\cC = \{C_i,\cM_{i,j}\}$ be an oriented flow category. Assume $\dim C_i \le k$ for all $i$. Then the cochain complex $C^*(\cC)$ only depends on $C_i$ and those $\cM_{i,j}$ with $\dim \cM_{i,j} \le 2k$.
\end{proposition}

\begin{remark}\label{rmk:MB}
	Roughly speaking, the part of the differential $D$ from $H^*(C_i)$ to $H^*(C_{i+k})$ is defined by the composition $t_*\circ s^*$ through $C_i\stackrel{s}{\leftarrow}\cM_{i,i+k}\stackrel{t}{\to} C_{i+k}$. However, since $\cM_{i,i+k}$ is not closed, $t_*\circ s^*$ is not well-defined on cohomology. In fact, after choosing representatives of $H^*(C_i)$ in $\Omega^*(C_i)$ (e.g.\ harmonic forms),  differential $D$ for a Morse-Bott flow category is given by $t_*\circ s^*$ on $\cM_{i,i+k}$ plus many correction terms from possible breakings of $\cM_{i,i+k}$:
	\begin{align}
	    	\int_{C_{i+k}} D \alpha \wedge \gamma =& \pm\int_{\cM_{i,i+k}}s^*\alpha \wedge t^*\gamma \nonumber \\ & +\lim_{n\to\infty}\sum_{0<j<k} \pm \int_{\cM_{i,i+j}\times \cM_{i+j,i+k}}s^*\alpha\wedge (t\times s)^*f^n_{i+j}\wedge t^* \gamma \label{eqn:int}\\
	    	& +\ldots, \nonumber
	\end{align}
	where $\alpha,\gamma$ are the chosen differential form representatives of elements in $H^*(C_i),H^*(C_{i+k})$ and $f^n_{i+j}$ is a $\dim C_{i+j}-1$ form on $C_{i+j}\times C_{i+j}$. The suppressed terms are integrations on products $\cM_{i,*}\times \ldots \times \cM_{*,i+k}$ with more $f^n_*$ inserted, see \cite{MB} for details. It is clear that Proposition \ref{prop:kchain} follows from \eqref{eqn:int}. Although formula \eqref{eqn:int} only depends on $\cM_{i,j}$ with $\dim \cM_{i,j}\le 2k$, the proof of $D^2=0$ requires the existence of higher dimensional ($\dim \le 4k+1$) moduli spaces.
\end{remark}

We call a flow category Morse iff $\dim C_i = 0$ for all $i$ and Morse-Bott otherwise. In the Morse case considered in Proposition \ref{prop:moduli}, since $f^n_i$ has degree $-1$, i.e.\ $f^n_i=0$, the cochain complex associated to $\cC^{R,J}$ is the usual Floer cochain complex generated by $\cC(H_R)\cup \cP^*(H)$ with differential solely contributed by zero dimensional moduli spaces
$$\rD x_i := \sum_{j}(\int_{\cM_{x_i,x_j}}1)x_j,$$
i.e.\ we count those moduli spaces $\cM_{x_i,x_j}$ of dimension $0$. Similarly, we have cochain complexes $C^*(\cC_0^{R,J}),C^*(\cC_+^{R,J})$ and a tautological short exact sequence of cochain complexes, 
\begin{equation}\label{eqn:short}
0\to C^*(\cC_0^{R,J}) \to C^*(\cC^{R,J}) \to C^*(\cC_+^{R,J})\to 0,
\end{equation}
as well as the truncated versions. Moreover, we have 
$$ C^*(\cC^{R,J}) = \varinjlim_i C^*(\cC^{R,J}_{\le i}), \quad  C^*(\cC^{R,J}_+) = \varinjlim_i C^*(\cC^{R,J}_{+,\le i}).$$
Since $J$ is time-independent on $W$, the cochain complex $C^*(\cC_0^{R,J})$ is the Morse cochain complex of $W$ for the Morse-Smale pair $(H_R, g:=\omega(\cdot, J\cdot))$. Hence we have $H^*(C^*(\cC_0^{R,J})) = H^*(W)$. Moreover, $H^*(C^*(\cC^{R,J}))$ is the symplectic cohomology $SH^*(W)$ and $H^*(C^*(\cC^{R,J}_+))$ is the positive symplectic cohomology $SH^*_+(W)$, see \cite{cieliebak2015symplectic,ritter2013topological,seidel2008biased} for more detailed discussion on those invariants. Then \eqref{eqn:short} gives rise to the tautological long exact sequence
$$\ldots \to H^*(W)\to SH^*(W) \to SH^*_+(W) \to H^{*+1}(W)\to \ldots$$ 
\begin{remark}
	Since we only consider contractible orbits in domains with vanishing first Chern class, the Conley-Zehnder index is well-defined in $\Z$ independent of all choices. Our grading convention follows \cite{ritter2013topological}, i.e.\ $|x_i| := n - \mu_{CZ}(x_i)$, where $\mu_{CZ}$ is the Conley-Zehnder index. Such convention implies that if $x_i$ is a critical point of the $C^2$-small Morse function $H_R|_W$, then $|x_i|$ equals to the Morse index. The convention here differs from \cite{seidel2008biased} by $n$. 
\end{remark}

\subsection{Continuation maps}
In this paper, we will only consider a special class of continuation maps, namely we only consider homotopies of almost complex structures and homotopies of Hamiltonians between $H_R$ for different $R$. Let $\rho(s)$ be a smooth non-decreasing function such that $\rho(s)=0,s\ll 0$ and $\rho(s)=1,s\gg 0$. Given $0<R_- \le R_+ \le 1$, we have a homotopy of Hamiltonians $H_{R_+,R_-}:=H_{\rho(s)R_++(1-\rho(s))R_-}:\R_s \times S^1 \times \widehat{W} \to \R$. Then we have the following properties for $H_{R_+,R_-}$.
\begin{enumerate}
	\item $H_{R_+,R_-}=H_{R_-}$ for $s\ll 0$ and is $H_{R_+}$ for $s \gg 0$.
	\item $\partial_s H_{R_+,R_-}\le 0$.
	\item $H_{R_+,R_-}$ outside $r=b_0$ does not depend on $s$.
\end{enumerate}
Then for $x \in \cC(H_{R_+})\cup \cP^*(H)$ and $y\in \cC(H_{R_-})\cup \cP^*(H)$, let $J_s$ be a homotopy of admissible almost complex structures, we use $\cH_{x,y}$ denote the compactified moduli space of solutions to 
$$\partial_s u + J_s(u-X_{H_{R_+,R_-}})=0, \lim_{s\to \infty} u =x, \lim_{s\to -\infty} u = y.$$
Then for generic choice of $J_s$, $\cH_{x,y}$ is a manifold with boundary and corners by an analogue of Proposition \ref{prop:moduli}. They give rise to a flow morphism in the following sense.

\begin{definition}[{\cite[Definition 3.18]{MB}}]
	An oriented flow morphism $\fH:\cC \Rightarrow \cD$ from an oriented flow category $\cC:=\{C_i,\cM^{C}_{i,j}\}$ to another oriented flow category $\cD:=\{D_i,\cM^{D}_{i,j}\}$ is a family of compact oriented manifolds with boundary and corners $\{\cH_{i,j}\}_{i,j\in \Z}$, such that the following holds. 
	\begin{enumerate}
		\item There exists $N\in \Z$,  such that when $i-j>N$, $\cH_{i,j}=\emptyset$.
		\item\label{m0} There are two smooth maps $s: \cH_{i,j}\to C_i, t:\cH_{i,j}\to D_j$..
		\item\label{m2} For every $i_0<i_1<\ldots<i_k$, $ j_0<\ldots<j_{m-1}<j_m$, the fiber product $\cM^C_{i_0,i_1}\times_{i_1}\ldots \times_{i_k} \cH_{i_k,j_0}\times_{j_0}\ldots\times_{j_{m-1}}\cM^D_{j_{m-1},j_m}$  is cut out transversely.
		\item\label{m1} There are smooth maps $m_L:\cM^C_{i,j}\times_j \cH_{j,k}\to \cH_{i,k}$ and $m_R: \cH_{i,j}\times_j \cM^D_{j,k}\to \cH_{i,k}$, such that 
		$$s\circ m_L(a,b)=s^C(a), \quad t\circ m_L(a,b) =t(b),$$
	    $$ s\circ m_R(a,b) =s(a),  \quad t\circ m_R(a,b)  =t^D(b),$$
		where map $s^C$ is the source map for flow category $\cC$ and map $t^D$ is the target map for flow category $\cD$.
		\item\label{m4}The map $m_L\cup m_R:\cup_j \cM^C_{i,j}\times_j \cH_{j,k}\cup_j \cH_{i,j}\times_j \cM^D_{j,k}\to \partial \cH_{i,k}$ is a diffeomorphism up to zero measure.
		\item The orientations of $\cH_{i,j}$ are compatible with orientations of $C_i,D_i,\cM^C_{i,j},\cM_{i,j}^D$ in the sense of \cite[(6) of Definition 3.18]{MB}.
	\end{enumerate}	
\end{definition} 
Therefore $\{\cH_{x,y}\}$ defines an oriented flow morphism $\fH^{R_+,R_-}$ from $\cC^{R_+,J_+}$ to $\cC^{R_-,J_-}$. By \cite[Theorem 3.21]{MB}, flow morphisms induce cochain maps between the the cochain complexes of the flow categories according to a formula similar to \eqref{eqn:int}. Hence in our situation here, $\fH^{R_+,R_-}$ is the geometric data required to define the continuation map. In the Morse case, the cochain map is defined by counting zero dimensional moduli spaces in $\{\cH_{x,y}\}$, which is indeed the classical continuation map. Since we have $\partial_s H_{R_+,R_-}\le 0$, then if $\cH_{x,y}\ne \emptyset$, we have $\cA_{H_{R_-}}(y)-\cA_{H_{R_+}}(x)\ge 0$. Therefore the flow morphism $\fH^{R_+,R_-}$ preserves the action filtration, in particular, the filtration induced by $W^i$. Hence we have the following flow morphisms.
\begin{enumerate}
	\item $\fH^{R_+,R_-}_0:\cC_0^{R_+,J_+} \Rightarrow \cC_0^{R_-,J_-}$.
	\item $\fH^{R_+,R_-}_+:\cC_+^{R_+,J_+}\Rightarrow \cC_+^{R_-,J_-}$.
	\item $\fH^{R_+,R_-}_{\le i}:\cC_{\le i}^{R_+,J_+}\Rightarrow \cC_{\le i}^{R_-,J_-}$.
	\item $\fH^{R_+,R_-}_{+,\le i}:\cC_{+,\le i}^{R_+,J_+}\Rightarrow \cC_{+,\le i}^{R_-,J_-}$.
\end{enumerate}

\subsection{Sphere bundles and Gysin exact sequence}
For any oriented $k$-sphere bundle $\pi:E \to W$ with $k$ odd, there is an associated Gysin exact sequence
\begin{equation}\label{eqn:classgysin}
	\to  H^{i}(W) \stackrel{\pi^*}{\longrightarrow} H^{i}(E) \stackrel{pi_*}{\longrightarrow} H^{i-k}(W) \stackrel{\wedge(-e)}{\longrightarrow}  H^{i+1}(W) \to
\end{equation}
Here $\pi_*$ is integration along the fiber using the convention in \cite[\S 6]{bott2013differential} and $e$ is the Euler class of $\pi$, the extra sign is for the consistency with \cite[Proposition 6.24]{MB}. In this subsection, we consider sphere bundles over symplectic cohomology and deduce the associated Gysin exact sequences. Such construction can be viewed as a higher dimensional analogue of Floer cohomology with local systems. Gysin exact sequence in Floer theory was first considered by Bourgeois-Oancea \cite{bourgeois2009exact}, where the exact sequence arises from a $S^1$-bundle in the construction of $S^1$-equivariant symplectic homology. Fiber bundles over Floer theory was considered by Barraud-Cornea \cite{barraud2007lagrangian}, where they considered the path-loop fibration. Smooth fiber bundles considered in this paper is technically easier to deal with. The construction in \cite{MB} works as long as the moduli spaces support integration \cite{hofer2010integration}. We first recall the concept of sphere bundles over flow categories.
\begin{definition}[{\cite[Definition 6.17]{MB}}]
	Let $\cC=\{C_i,\cM^C_{i,j}\}$ be an oriented flow category. An oriented $k$-sphere bundle over $\cC$ is a flow category $\cE=\{E_i,\cM^E_{i,j}\}$ with functor $\pi: \cE\to\cC$, such that the following holds.
	\begin{enumerate}
		\item $\pi$ maps $E_i$ to $C_i$ and $\cM^E_{i,j}$ to $\cM^C_{i,j}$.
		\item $\pi:E_i\to C_i$ and $\pi:\cM^E_{i,j}\to \cM^C_{i,j}$ are oriented sphere bundles such that both bundle maps $s^E_{i,j},t^E_{i,j}$ preserve the orientation. 
	\end{enumerate} 
\end{definition}
By \cite[Proposition 6.18]{MB}, an oriented $k$-sphere bundle $\cE$ over an oriented flow category is an oriented flow category. The construction \cite[Definition 3.8]{MB} assigns $\cE$ to a cochain complex, and we have the following.
\begin{proposition}[{\cite[Theorem 6.19]{MB}}]\label{prop:gysin}
	Let $\cE$ be an oriented $k$-sphere bundle over an oriented flow category $\cC$. Then we have a short exact sequence of cochain complexes\footnote{To be more precise, we have a short exact sequence using certain choices in the construction. However, in the special case that $\cC$ is Morse, the minimal construction in \cite[Theorem 3.10]{MB}, i.e.\ the one in Remark \ref{rmk:MB},  gives the short exact sequence.}
	$$0\to C^*(\cC) \stackrel{\pi^*}{\to} C^*(\cE) \stackrel{\pi_*}{\to}\cC^{*-k}(\cC) \to  0.$$
	It induces the following Gysin exact sequence
	\begin{equation}\label{eqn:gysin}
	\ldots \to H^{*}(\cC) \to H^*(\cE) \to H^{*-k}(\cC) \to H^{*+1}(\cC) \to \ldots.
	\end{equation}
\end{proposition}
\begin{remark}
	$\pi^*,\pi_*$ are induced by oriented flow morphisms, which are completely determined by $\cE$. Here we give an explanation in the special case when $E\to C$ is an actual sphere bundle. The compact  manifold $C$ can be understood as a flow category whose object space is diffeomorphic to $C$ and morphism space consists of only identify morphisms. Then $E$ can be understood as a sphere bundle over the flow category $C$. Then $\pi^*$ is given by the flow morphism $C \stackrel{s = \pi}{\longleftarrow} E \stackrel{t = \Id}{\longrightarrow} E$ and $\pi_*$ is given by the flow morphism $E \stackrel{s = \Id}{\longleftarrow} E \stackrel{t = \pi}{\longrightarrow} C$. In particular, $\pi^*$ is the composition  $t_*\circ s^*$ from $C \stackrel{s = \pi}{\longleftarrow} E \stackrel{t = \Id}{\longrightarrow} E$, which is indeed the pullback $\pi^*$ on cohomology, and $\pi_*$ is the composition $t_*\circ s^*$ from  $E \stackrel{s = \Id}{\longleftarrow} E \stackrel{t = \pi}{\longrightarrow} C$, which is the pushforward $\pi_*$ on cohomology. In general, the underlying flow morphisms of $\pi^*,\pi_*$ are induced from the identity flow morphism of $\cE$ \cite[Definition 3.23]{MB}.
\end{remark}
\begin{remark}[{\cite[Corollary 6.23]{MB}}]\label{rmk:depend}
	Assume $\cC=\{C_i,\cM_{i,j}\}$ is a Morse flow category, i.e.\ $\dim C_i = 0$. Then $H^*(\cE)$ and the Gysin exact sequence only depend on $\cM^E_{i,j}$ with $\dim \cM^E_{i,j} \le 2k$. In particular, we only use moduli spaces $\cM_{i,j}$ of dimension up to $k$. The non-triviality of higher dimensional moduli spaces $\cM_{i,j}$ is the foundation of the existence of interesting sphere bundles. Although the formula only requires $\cM^{E}_{i,j}$ of dimension up to $2k$,  we need  a priori the existence of the full flow category to guarantee the existence of Gysin sequences. 
\end{remark}
\begin{remark}
	The Gysin exact sequence considered in \cite{MB} works for any Morse-Bott flow category $\cC$. In the case considered here, i.e.\ $\cC$ is Morse, it is possible to generalize the construction in \cite{hutchings2017axiomatic} to the $S^k$ case to get a $\Z$-coefficient Gysin exact sequence. 
\end{remark}
We call a Gysin exact sequence \eqref{eqn:gysin} trivial iff the Euler part $H^{*}(\cC) \to H^{*+k+1}(\cC)$ is zero. In the case we consider, a sphere bundle over the Liouville domain will induce a sphere bundle over the symplectic flow category.
\begin{proposition}\label{prop:sphere}
	Let $W$ be a Liouville domain and $J \in \cJ^{R}(W)$. Let $\pi: E \to W$ be an oriented $k$-sphere bundle and $P_\gamma$ the parallel transport along path $\gamma$ for a fixed connection on $E$. Then we have an oriented $k$-sphere bundles $\cE^{R,J},\cE^{R,J}_0,\cE^{R,J}_+,\cE^{R,J}_{\le i}, \cE^{R,J}_{+,\le i}$  over $\cC^{R,J},\cC^{R,J}_0,\cC^{R,J}_+,\cC^{R,J}_{\le i}, \cC^{R,J}_{+,\le i}$  respectively.
\end{proposition}
\begin{proof}
	If we write $\cC = \{x_i,\cM_{i,j}\}$, then we define $E_i:= E_{x_i(0)}\simeq S^k$ and $\cM^E_{i,j} := \cM_{i,j}\times E_i$. The structure maps are defined by
	$$
	\begin{array}{rclrcl}
	s^E:\cM_{i,j}\times E_i &\to& E_i,\\
	     (u,v) &\mapsto& v;\\
	t^E:\cM_{i,j}\times E_i &\to& E_j, \\
	(u,v) &\mapsto& P_{u(-\cdot,0)}v;\\
	m: (\cM_{i,j}\times E_i)\times_{E_j}(\cM_{j,k}\times E_j) &\to& \cM_{i,k}\times E_i, \\
	(u_1, v, u_2, P_{u_1(-\cdot,0)}v) &\mapsto& (u_1,u_2,v).
	\end{array}
	$$
	It is direct to check that they form a category. The fiber product transversality follows from that $s^E,t^E$ are submersive. Because $E\to W$ is an oriented sphere bundle, we have $E_i=E_{x_{i}(0)}$ is oriented and $P_\gamma$ preserves the orientation. Hence $\cE^{R,J} = \{E_i,\cM^E_{i,j}\}$ is an oriented $k$-sphere bundle over $\cC^{R,J}$. Similarly for other flow categories.
\end{proof}

\begin{example}
	To further explain Remark \ref{rmk:depend}, we can look at two flow categories: $\obj_{\cC_1}$ is set of two points $\{x_0,x_1\}$ with $\cM_{0,1}=\emptyset$,  $\obj_{\cC_2}=\{x_0,x_1\}$ while $\cM_{0,1}=S^1$. $\cC_2$ can be viewed the flow category associated to the Morse theory of the height function on $S^2$. Then $\cC_1$ does not admits any nontrivial $S^n$ bundle, in particular, the associated Euler part is always trivial.  Even though $\cC_1,\cC_2$ have the same cohomology of rank $2$, $\cC_2$ admits a nontrivial $S^1$ bundle $\cE_2=\{S^1_0,S^1_1,\cM^E_{0,1}=S^1\times S^1\}$, where $\cM^E_{0,1}$ is viewed as a $S^1$ bundle over the second factor $S^1$, which is viewed as $\cM_{0,1}$. The structural maps are defined as $s^E:(\theta,t)\mapsto \theta, t^E:(\theta,t)\mapsto \theta+t$. One may check the induced Gysin exact sequence has non-trivial Euler part. Indeed, $\cE_2$ is the $S^1$ bundle induced from the Hopf fibration over $S^2$ using an appropriate parallel transport. This example shows that higher dimensional moduli spaces are foundations for interesting fibrations.
\end{example}

Similarly, there is a notion of oriented sphere bundles over flow-morphisms. Assume given two oriented $k$-sphere bundles $\cE \to \cC, \cF \to \cD$. Let $\fH:\cC \Rightarrow \cD$ be an oriented flow morphism. Then a $k$-sphere bundle $\fP$ over $\fH$ is defined as follows.
\begin{enumerate}
	\item $\fP=\{\cP_{i,j}\}$ is a flow morphism from $\cE$ to $\cF$.
	\item $\pi:\cP_{i,j} \to \cH_{i,j}$ is a $k$-sphere bundle, such that $s^P,t^P$ are bundles maps covering $s^H,t^H$.
	\item $\pi: \cP_{i,j} \to \cH_{i,j}$ is an oriented bundle and $s^P,t^P$ preserve the orientation.
\end{enumerate}
Given a sphere bundle $E\to W$ with a parallel transport, let $\fH,\fH_0,\fH_+, \fH_{\le i}, \fH_{+,\le i}$ be the flow morphisms constructed from the homotopy $H_{R_+,R_-}$. Then by the same construction in Proposition \ref{prop:sphere}, there are induced oriented sphere bundles $\fP,\fP_0,\fP_+,\fP_{\le i}, \fP_{+,\le i}$ over them. Moreover, the parallel transport at two ends can be different. In this case, we need to fix a smooth family of connections $\xi_s$ such that $\xi_s$ is the connection for the negative end for $s\ll 0$ and $\xi_s$ is the connection for the positive end for $s\gg 0$. Then given a Floer solution $u(s,t)$ in the flow morphism for continuation maps, the structure maps for the sphere bundle is defined using the parallel transport w.r.t.\ $\xi_s$ over $u(s,0)$.

By \cite[Proposition 6.27]{MB}, sphere bundles over flow morphisms induce morphisms of Gysin sequences. We define $\cJ^{R}_{\le i}$ to be the set of almost complex structures such that the flow category $\cC^{R,J}_{\le i}$ is defined. Given a sequence of real numbers $1>R_1>R_2>\ldots > 0$ and a sequences of almost complex structures $J_i$, such that $J_i\in \cJ^{R_i}_{\le i}(W)$, if we fix any oriented $S^k$ bundle $E\to W$ along with a connection, then Proposition \ref{prop:gysin} induces the following commutative diagram of exact sequences.
\begin{equation}\label{eqn:diagram}
\resizebox{12cm}{!}{
\xymatrix{
	 & \ar[d] & \ar[d] & \ar[d] & \ar[d] \\
	 \ar[r] & \varinjlim_i H^{*+k}(\cE_0^{R_i,J_i}) \ar[r]\ar[d] & \varinjlim_i H^*(\cC_0^{R_i,J_i}) \ar[r]\ar[d] & \varinjlim_i H^{*+k+1}(\cC_0^{R_i,J_i}) \ar[r]\ar[d] & \varinjlim_i H^{*+k+1}(\cE_0^{R_i,J_i}) \ar[r]\ar[d] & \\
	 \ar[r] & \varinjlim_i H^{*+k}(\cE^{R_i,J_i}) \ar[r]\ar[d] & \varinjlim_i H^*(\cC^{R_i,J_i}) \ar[r]\ar[d] & \varinjlim_i H^{*+k+1}(\cC^{R_i,J_i}) \ar[r]\ar[d] & \varinjlim_i H^{*+k+1}(\cE^{R_i,J_i}) \ar[r]\ar[d] & \\
	 \ar[r] & \varinjlim_i H^{*+k}(\cE_+^{R_i,J_i}) \ar[r]\ar[d] & \varinjlim_i H^*(\cC_+^{R_i,J_i}) \ar[r]\ar[d] & \varinjlim_i H^{*+k+1}(\cC_+^{R_i,J_i}) \ar[r]\ar[d] & \varinjlim_i H^{*+k+1}(\cE_+^{R_i,J_i}) \ar[r]\ar[d] & \\
	 & & & &}
	 }
\end{equation}

Note that $\varinjlim_i H^*(\cC_0^{R_i,J_i})=H^*(W)$, $ \varinjlim_i H^*(\cC^{R_i,J_i})=SH^*(W)$ and  $\varinjlim_i H^*(\cC_+^{R_i,J_i})=SH^*_+(W)$. We expect $\varinjlim_i H^*(\cE^{R_i,J_i}),\varinjlim_i H^*(\cE_+^{R_i,J_i})$ are also well-defined objects. But this requires proving invariance under changing various defining data like $H_R, R_i,J_i$ and the parallel transport $P$. In the Morse-Bott situation considered here, we need to use the flow-homotopy introduced in \cite[Definition 3.29]{MB} to prove the invariance. However, for the purpose of this paper, we do not need a well-defined Floer theory for the sphere bundle and are only interested in the Euler part. We will proceed with this version involving all specific choices. We will suppress the choice of parallel transport for simplicity, and only specify our choice when it matters.

Since the constant orbits part corresponds to the Morse theory on $W$, there the Gysin sequence should be the regular Gysin sequence.
\begin{proposition}[{\cite[Theorem 8.14]{MB}}]\label{prop:MorseGysin}
	The Gysin sequence 
	$$
	\to \varinjlim_i H^{*+k}(\cE_0^{R_i,J_i}) \to \varinjlim_i H^*(\cC_0^{R_i,J_i})\to \varinjlim_i H^{*+k+1}(\cC_0^{R_i,J_i})\to \varinjlim_i H^{*+k+1}(\cE_0^{R_i,J_i}) \to 
	$$
	is the classical Gysin exact sequence \eqref{eqn:classgysin} for $\pi:E\to W$.
\end{proposition}
By the Gysin exact sequence for symplectic cohomology $SH^*(W)$, we have the following vanishing result.
\begin{proposition}\label{prop:vanish}
	If $SH^*(W) = 0$ and $E$ is an oriented sphere bundle over the Liouville domain $W$, then $\varinjlim_i H^{*}(\cE^{R_i,J_i})=0$ for any defining data.
\end{proposition}

\subsection{Naturality}
In the neck-stretching argument, we need to compare moduli spaces of two fillings, hence naturality is important. Moreover, we can only get the moduli spaces appear in the counting matched up for two fillings, i.e.\ moduli spaces of dimension up to $k$. But to apply Proposition \ref{prop:gysin}, we need the full flow category. In particular, it is possible that the higher dimensional moduli spaces are not cut out transversely in the neck-stretching. In the following, we discuss those aspects in a similar way to \cite{zhou2019symplectic}. 

\begin{definition}
$\cJ^{R,\le k}(W)\subset \cJ(W)$ is the set of admissible almost complex structures, such that moduli spaces of $H_R$ up to dimension $k$ is cut out transversely.  $\cJ^{R,\le k}_{+}(W)$ stands for the positive version and $\cJ^{R,\le k}_{\le i}(W)$, $\cJ^{R,\le k}_{+,\le i}(W)$ are the truncated versions.
\end{definition}
All above sets are of second Baire category. Moreover, as a consequence of compactness, $\cJ^{R,\le k}_{(+),\le i}$ is open and dense. The following is a standard result in Floer theory.
\begin{proposition}\label{prop:natural}
	Let $J_0\in \cJ^{R_0,\le 0}_{+,\le i}(W),J_1\in \cJ^{R_1,\le 0}_{+,\le i+1}(W)$ for $R_0>R_1$, then the continuation map $H^*(\cC^{R_0,J_0}_{+,\le i})\to H^*(\cC^{R_1,J_1}_{+,\le i+1})$ is independent of the homotopy of almost complex structures.
\end{proposition}
We also recall the following result from \cite{zhou2019symplectic}. 
\begin{proposition}[{\cite[Lemma 2.15]{zhou2019symplectic}}]\label{prop:natural2}
	Let $J_s, s\in [0,1]$ be a smooth path in $\cJ(W)$ and $R_s$ be a non-increasing function in $(0,1]$, such that $J_s\in \cJ^{R_s,\le 0}_{+,\le i}(W)$, then the continuation map $C^*(\cC^{R_0,J_0}_{+,\le i})\to C^*(\cC^{R_1,J_1}_{+,\le i})$  is homotopic to identity\footnote{Note that generators are the same for $(H_{R_0},J_0)$ and $(H_{R_1},J_1)$, hence identity map makes sense.}.
\end{proposition}
Note that we assume $H_R$ stays the same outside $r=b_0$ for any $R$, i.e.\ the generators for positive symplectic cohomology stay the same. The same argument of \cite[Lemma 2.15]{zhou2019symplectic} can be applied here for positive symplectic cohomology, even though we assume $H=0$ on $W$ in \cite[Lemma 2.15]{zhou2019symplectic}. 

Although the full flow-category requires transversality for all moduli spaces, the Gysin sequence is well-defined for almost complex structure of low regularity by the following.
\begin{proposition}\label{prop:natural1}
	Let $E\to W$ be a $k$-sphere bundle, then the Euler part of the Gysin exact sequence
	$$ \to \varinjlim_i H^{*+k}(\cE^{R_i,J_i}_{+,\le i}) \to \varinjlim_i H^*(\cC^{R_i,J_i}_{+,\le i})\to  \varinjlim_i H^{*+k+1}(\cC^{R_i,J_i}_{+,\le i}) \to  \varinjlim_i H^{*+k+1}(\cE^{R_i,J_i}_{+,\le i}) \to 
	$$
	is well defined for $J_i\in \cJ^{R_i,\le k}_{+,\le i}(W)$.
\end{proposition}
\begin{proof}
	We first prove the truncated Gysin sequence
	$$
		\to  H^{*+k}(\cE^{R_i,J_i}_{+,\le i})\to  H^*(\cC^{R_i,J_i}_{+,\le i}) \to  H^{*+k+1}(\cC^{R_i,J_i}_{+,\le i}) \to H^{*+k+1}(\cE^{R_i,J_i}_{+,\le i}) \to
	$$
	is defined. Since we can find an open neighborhood $\cU \subset \cJ^{R_i,\le k}_{+,\le i}(W)$ of $J_i$, we have a universal moduli space $\cup_{J\in \cU} M_{x,y,J}$, where $M_{x,y,J}$ is the moduli space of \emph{unbroken} Floer trajectories using $J$ in \eqref{eqn:floer} for the positive symplectic cohomology for $x,y\subset W^i$.  The universal moduli space is a Banach manifold and its projection to $\cU$ is regular. For each $J\in \cU\cap \cJ^{R_i}_{+,\le i}(W)$, we have a flow category with sphere bundle. We use $d_J$ to denote the differential on the cochain complex of the sphere bundle. Moreover, $d_J$ is well-defined by \eqref{eqn:int} for $J\in \cU$, even though $d_J^2$ may not be zero a priori unless $J\in \cU \cap \cJ^{R_i}_{+,\le i}(W)$. Since the integration \eqref{eqn:int} only depends on the full measure set $M_{x,y,J}$. We have $d_J$ varies continuously\footnote{The compactification $\cM_{x,y}$ also varies continuously for $J\in\cU \subset \cJ^{R,\le k}_{+,\le i}$ when $\dim \cM_{x,y}\le k$.} over $\cU$. Since $\cU\cap \cJ^{R_i}_{+,\le i}(W)$ is dense in $\cU$, we have $d_J^2=0$ for every $J\in \cU$. As a consequence, the Gysin sequence is defined for every $J\in \cU$, in particular for $J_i$. 	Then by the similar argument, by finding $J_{i,i+1}\in \cJ^{R_i,R_{i+1},\le k}_{+,\le i}(W)$, we have a commutative diagram of the truncated Gysin sequence. which yields a Gysin sequence of the direct limit. Since we only need the well-definedness of the Euler part, the continuation map $H^*(\cC^{R_i,J_i}_{+,\le i})\to H^*(\cC^{R_{i+1},J_{i+1}}_{+,\le i+1})$ is independent of the choice of $J_{i,i+1}$ by Proposition \ref{prop:natural}, the claim follows.
\end{proof}

\begin{corollary}\label{cor:natural}
		Let $E\to W$ be a $k$-sphere bundle. Assume $J_s, s\in [0,1]$ is a smooth path in $\cJ(W)$ and $R_s$ is a non-increasing smooth function taking values in $(0,1]$, such that $J_s\in \cJ^{R_s,\le k}_{+,\le i}(W)$. Then the following Euler parts of the Gysin exact sequences are commutative 
	    $$
	    \xymatrix{
	    H^*(\cC^{R_0,J_0}_{+,\le i})\ar[r]\ar[d] & H^{*+k+1}(\cC^{R_0,J_0}_{+,\le i})\ar[d]\\
        H^*(\cC^{R_1,J_1}_{+,\le i})\ar[r] & H^{*+k+1}(\cC^{R_1,J_1}_{+,\le i})
       }
	    $$
		where the vertical arrows are the continuation maps, which are homotopic to identity by Proposition \ref{prop:natural2}. 
\end{corollary}
\begin{proof}
	Assume in addition that $J_0\in \cJ^{R_0}_{+, \le i}, J_1\in \cJ^{R_1}_{+, \le i}$, then we can find a regular enough homotopy from $J_0$ to $J_1$, such that we have a flow-morphism between the associated flow categories. The induced continuation map induces an isomorphism on the Euler part of the Gysin exact sequence. By Proposition \ref{prop:natural2}, the continuation map $H^*(\cC^{R_0,J_0}_{+,\le i})\to H^*(\cC^{R_1,J_1}_{+,\le i})$ is the identity. Therefore the Euler part of the Gysin sequence are the same for $J_0,J_1$. Since $\cJ^{R_*,\le k}_{+,\le i}(W)$ is open and contains $\cJ^{R_*}_{+,\le i}(W)$ as a dense set. Then the argument in Proposition \ref{prop:natural1} shows that the Euler part varies continuously with respect to $J$. Then the claim follows.
\end{proof}

\begin{proposition}\label{prop:final}
For $J_i\in \cJ^{R_i,\le k}_{+,\le i}(W)$, we have the following well-defined commutative digram 
$$
\xymatrix{\varinjlim_i H^*(\cC^{R_i,J_i}_{+,\le i}) \ar[rr]\ar[d] & & \varinjlim_i  H^{*+k+1}(\cC^{R_i,J_i}_{+,\le i})\ar[d]\\
	H^{*+1}(W)\ar[rr]^{\wedge(-e(E))} & &H^{*+k+2}(W)
	}
$$
where the horizontal map is the Euler part, and the vertical map is the connecting map from the positive symplectic cohomology to the cohomology of the filling.
\end{proposition}
\begin{proof}
Since $\cJ^{R_i,\le k}_{+,\le i}(W)$ is open, we can choose $J'_i$ in a connected neighborhood of $J_i$ in $\cJ^{R_i,\le k}_{+,\le i}(W)$, such that $J_i'\in \cJ^{R_i}_{\le i}(W)$. Then by Corollary \ref{cor:natural}, we have the following commutative diagram. 
$$
\xymatrix{
	\varinjlim_i H^*(\cC^{R_i,J_i}_{+,\le i}) \ar[rr]\ar[d]^{=} & &\varinjlim_i  H^{*+k+1}(\cC^{R_i,J_i}_{+,\le i})\ar[d]^{=}\\
	\varinjlim_i H^*(\cC^{R_i,J'_i}_{+,\le i}) \ar[rr]\ar[d] & &\varinjlim_i  H^{*+k+1}(\cC^{R_i,J'_i}_{+,\le i})\ar[d]\\
	H^{*+1}(W)\ar[rr]^{\wedge(-e(E))} & &H^{*+k+2}(W)
}
$$
By \cite[Proposition 2.17]{zhou2019symplectic}, the vertical arrows in the bottom square do not depend on the choice of $J'_i$, the claim follows.
\end{proof}

\subsection{Neck-stretching and independence of the positive Gysin sequence}
Let $(Y,\alpha)$ be a $k$-ADC contact manifold with two topologically simple fillings $W_1,W_2$. Note that $\widehat{W}_1, \widehat{W}_2$ both contain the symplectization $(Y\times (0,\infty)_r, \rd(r\alpha))$. Since $Y$ is $k$-ADC, there exist nested contact type surfaces $Y_i\subset Y\times (0,1)$, such that $Y_i$ lies outside of $Y_{i+1}$ and contractible Reeb orbits of contact form $r\alpha|_{Y_i}$ has the property that the degree is greater than $k$ if the  period is smaller than $D_i$.  

\begin{figure}[H]
	\begin{center}
		\begin{tikzpicture}[yscale=0.5]
		\draw (0,0) to [out=90, in =200] (7,3);
		\draw (7,3) to [out=20, in=190] (9,3.5);
		\draw (0,0) to [out=270, in =160] (7,-3);
		\draw (7,-3) to [out=340, in=170] (9,-3.5);
		\draw (6,2.73) to [out=300, in=60](6,-2.73);
		\draw (1,1.76) to [out=300, in=60](1,-1.76);
		\draw (4,2.45) to [out=330, in=90] (4.5,1) to [out=270, in=90] (3.8,0) to  [out=270, in=90](5,-1) to [out=270, in=60](4,-2.45);
		\draw (2.5,2.25) to [out=330, in=90](4,1.2) to [out=270, in=90](3,0) to [out=270, in=90](2,-1.2) to [out=270, in=120] (2.5,-2.25);
		\node at (4.2,0) {\footnotesize$Y_i$};
		\node at (2.6,0) {\footnotesize$Y_{i+1}$};
		\node at (7.7,0) {\footnotesize$Y$};
		\draw [decorate,decoration={brace,amplitude=10pt,raise=4pt},yshift=0pt]
		(1,1.76) -- (6,2.73) node [black,midway,yshift=0.6cm] {\footnotesize$Y\times (0,1)$};
		\end{tikzpicture}
	\end{center}
	\caption{$Y_i\subset \widehat{W}_*$}
\end{figure}
Neck-stretching near $Y_i$ is given by the following. Assume domains in the form of $Y_i\times [1-\epsilon_i,1+\epsilon_i]_{r_i}$ are disjoint for some small $\epsilon_i$, where $r_i$ is the coordinate determined by the Liouville vector field near $Y_i$ such that $r_i|_{Y_i}=1$. Assume $J|_{Y_i\times [1-\epsilon_i,1+\epsilon_i]_{r_i}}=J_0$, where $J_0$ is independent of $S^1$ and $r_i$,  and $J_0(r_i\partial_{r_i})=R_i,J_0\xi_i=\xi_i$,  where $\xi_i=\ker r\alpha|_{Y_i}$ and $R_i$ is the associated Reeb vector field. Then we pick a family of diffeomorphisms $\phi_R:[(1-\epsilon_i)e^{1-\frac{1}{R}}, (1+\epsilon_i)e^{\frac{1}{R}-1}]\to [1-\epsilon_i,1+\epsilon_i]$ for $R\in (0,1]$ such that $\phi_1=\Id$ and $\phi_R$ near the boundary is linear with slope $1$. Then the stretched almost complex structure $NS_{i,R}(J)$ is defined to be $J$ outside $Y_i\times [1-\epsilon_i,1+\epsilon_i]$ and is $(\phi_R\times \Id)_*J_0$ on $Y_i\times [1-\epsilon_i,1+\epsilon_i]$. Then $NS_{i,1}(J)=J$ and $NS_{i,0}(J)$ gives almost complex structures on the completions of the cobordism $X_i$ between $Y$ and $Y_i$, the filling of $Y_i$, and  the symplectization $Y_i\times \R_+$.

Since we need to stretch along different contact surfaces, we assume the $NS_{i,R}(J)$ have the property that $NS_{i,R}(J)$ will modify the almost complex structure near $Y_{i+1}$ to a cylindrical almost complex structure for $R$ from $1$ to $\frac{1}{2}$ and for $R\le \frac{1}{2}$, we only keep stretching along $Y_i$. We use $\cJ^{\le k}_{reg,SFT,\le i}(H_0)$ to denote the set of admissible regular $J$, i.e.\ almost complex structures satisfying Definition \ref{def:admissble} on the completion of $W$ outside $Y_i$ and asymptotic (in a prescribed way as in the stretching process) to $J_0$ on the negative cylindrical end, such that the following moduli space

$$\left\{ u:\R\times S^1\backslash Z \to \widehat{X_i}\left|\begin{array}{l}
\partial_su+J(\partial_tu-X_{H_0})=0 ,\\
\displaystyle \lim_{s\to \infty} u = x, \lim_{s\to \infty} u =y, \\
\cA_{H_0}(x),\cA_{H_0}(y)>-D_i,\\
Z=\{z_1,\ldots,z_I\}, \\
\displaystyle \lim_{z\to z_j} u =\gamma_j\times \{-\infty\}, \forall 1\le j \le I.
\end{array}     \right. \right\}/\R$$
up to dimension $k$ is cut out transversely. Here $\gamma_j$ is a Reeb orbit on $Y_i$ and $\lim_{z\to z_j} u =\gamma_j\times \{-\infty\}$ is a shorthand for $u$ asymptotic to $\gamma_j$ near the negative puncture $z_j\in \R\times S^1$. Then $\cJ^{\le k}_{reg,SFT,\le i}(H_0)$ is an open dense subset of all admissible almost complex structures on $\widehat{X_i}$. 
\begin{figure}[H]
    \begin{center}
	\begin{tikzpicture}[scale=0.7]
	\draw[rounded corners](0,0) -- (0,4) -- (2,4) -- (2,6);
	\draw[rounded corners](3,6)--(3,5)--(5,5)--(5,3);
	\draw[rounded corners](4,3)--(4,4)--(3,4)--(3,0);
	\draw[rounded corners](2,0)--(2,3)--(1,3)--(1,0);
	
	\draw (0,0) to [out=270,in=180] (0.5,-0.25);
	\draw (0.5,-0.25) to [out=0,in=270] (1,0);
	\draw[dotted] (0,0) to [out=90,in=180] (0.5,0.25);
	\draw[dotted] (0.5,0.25) to [out=0,in=90] (1,0);
	
	\draw (2,0) to [out=270,in=180] (2.5,-0.25);
	\draw (2.5,-0.25) to [out=0,in=270] (3,0);
	\draw[dotted] (2,0) to [out=90,in=180] (2.5,0.25);
	\draw[dotted] (2.5,0.25) to [out=0,in=90] (3,0);
	
	\draw (4,3) to [out=270,in=180] (4.5,2.75);
	\draw (4.5,2.75) to [out=0,in=270] (5,3);
	\draw[dotted] (4,3) to [out=90,in=180] (4.5,3.25);
	\draw[dotted] (4.5,3.25) to [out=0,in=90] (5,3);
	
	\draw (2,6) to [out=270,in=180] (2.5,5.75);
	\draw (2.5,5.75) to [out=0,in=270] (3,6);
	\draw (2,6) to [out=90,in=180] (2.5,6.25);
	\draw (2.5,6.25) to [out=0,in=90] (3,6);
	
	\node at (2.5,6) {$x$};
	\node at (4.5,3) {$y$};
	\node at (0.5,0) {$\gamma_1$};
	\node at (2.5,0) {$\gamma_2$};
	
	\draw[dashed] (-0.5,0)--(0,0);
	\draw[dashed] (1,0)--(2,0);
	\draw[dashed] (3,0)--(8,0);
	\node at (6,0.2) {$Y_i\times \{0\}$};
	
	\end{tikzpicture}
	\end{center}
	\caption{Moduli spaces for the definition of  $\cJ^{\le k}_{reg,SFT,\le i}(H_0)$.}\label{fig:SFT}
\end{figure}
To compare moduli spaces for two Liouville fillings $W_1,W_2$,  we can assume in the following that  $H_R$ outside $Y_i$ is the same for $W_1,W_2$ whenever $R\le\frac{1}{i}$. The following is simply a variant of \cite[Proposition 3.12]{zhou2019symplectic}.

\begin{proposition}\label{prop:neck}
	With setups above, there exist admissible $J^1_*,J^2_*$ on $\widehat{W}_*$ for $*=1,2$ and positive real numbers $\epsilon_1, \epsilon_2,\ldots \le 1$ and $\delta_1,\delta_2,\ldots \le 1$ with $\delta_i\le \frac{1}{i}$,  such that the following holds.
	\begin{enumerate}
		\item\label{ns1} For $R<\epsilon_i$ and any $R'\in [0,1]$, we have $NS_{i,R}(J^i_*)\in \cJ^{R\delta_i,\le k}_{+,\le i}(W_*), NS_{i+1,R'}(NS_{i,R}(J^i_*))\in \cJ^{R'R\delta_i, \le k}_{+,\le i}(W_*)$. Moreover all moduli spaces $\cM_{x,y}$ of dimension up to $k$ are the same for both $W_1,W_2$ and contained outside $Y_i$ for $x,y\in \cP^*(H)$ with action $\ge - D_i$.   
		\item\label{ns2} We have $J^{i+1}_*=NS_{i,\frac{\epsilon_i}{2}}(J^i_*)$ on $W^i_*$ and $\delta_{i+1}=\frac{\epsilon_i\delta_i}{2}$. 
	\end{enumerate}
\end{proposition}
\begin{proof}
		We prove the proposition by induction. Firstly, we set $\delta_1=1$. We then choose a $J^1$ such that $NS_{1,0}(J^1)\in \cJ^{\le k}_{reg,SFT,\le 1}(H_0)$. We will apply neck-stretching to $J^1$ at $Y_1$. Note that we need to arrange the Hamiltonian converging to constant near $Y_i$. We consider moduli space $\cM_{x,y,H_R}$ with expected dimension at most $k$ for $NS_{1,R}(J^1)$. Assume $\cM_{x,y,H_R}$ is not contained outside $Y_1$ in the stretching process. Then a limit curve $u$ outside $Y_1$ has one component by \cite[Lemma 2.4]{cieliebak2015symplectic}\footnote{Note that our symplectic action has the opposite sign compared to \cite[Proposition 9.17]{cieliebak2015symplectic}.}. Moreover, by the argument in \cite[Lemma 2.4]{cieliebak2015symplectic}, $u$ can only be asymptotic to Reeb orbits $\{\gamma_i\}_{i\in I}$ that are contractible in $W_*$ on $Y_1$ with period smaller than $D_1$. Since $W_*$ is topological simple, $\{\gamma_i\}_{i\in I}$ are contractible in $Y_1$. In particular, all of them have well-defined $\Z$-valued Conley-Zehnder indices with SFT degree $>k$. The expected dimension of the moduli spaces of such $u$ is $\ind(u)-1=|y|-|x|-\sum_{i\in I}(\mu_{CZ}(\gamma_i)+n-3)-1 <|y|-|x|-1-k<0$. Since $NS_{1,0}(J^1)\in \cJ^{\le k}_{reg,SFT,\le 1}(H_0)$, we have such $u$ is cut transversely. In particular, there is no such $u$ as the expected dimension is negative. Then for $R\ll 1$, we have $\cM_{x,y,H_R}$ using $NS_{1,R}(J^1)$ is contained outside $Y_1$, whenever $\dim \cM_{x,y,H_R}\le k$. Then $NS_{1,0}(J^1)\in \cJ^{\le k}_{reg,SFT,\le 1}(H_0)$ also implies that $NS_{1,R}(J^1)\in \cJ^{R,\le k}_{+,\le 1}(W_*)$ by the openness of transversality. 
		
		Next we will apply neck-stretching both at $Y_1,Y_2$. By the same argument as above,  for every $R'\in [0,1]$, we can find $\epsilon_{R'}>0$ and $\delta_{R'}>0$ such that for $\epsilon < \epsilon_{R'}$ and $|\delta-R'|<\delta_{R'}$,
	    \begin{enumerate}
			\item $NS_{2,\delta}(NS_{1,\epsilon}(J^1))\in \cJ^{\delta\epsilon,\le k}_{+,\le 1}(W_*)$.
			\item $\cM_{x,y,H_{\epsilon\delta}}$  is contained outside $Y_1$ if the expected dimension is at most $k$.
		\end{enumerate}
		Then the compactness of $[0,1]_{R'}$ implies that there exists $\epsilon_1>0$, such that for $R<\epsilon_i$ and any $R'\in [0,1]$, we have $NS_{1,R}(J^i_*)\in \cJ^{R,\le k}_{+,\le 1}(W_*)$ and $NS_{2,R'}(NS_{1,R}(J^1_*))\in \cJ^{R'R,\le k}_{+,\le 1}(W_*)$. Moreover, the moduli space $\cM_{x,y,H_{R'R}}$ for $NS_{2,R'}(NS_{1,R}(J^1_*))$ is contained outside $Y_1$. We can certainly arrange $\epsilon_1$ small enough, such that $\delta_2=\frac{\epsilon_1\delta_1}{2}=\frac{\epsilon_1}{2}\le \frac{1}{2}$. 
		Since moduli spaces in Figure \ref{fig:SFT} for $NS_{2,0}(NS_{1,R}(J^1_*))$ must be contained outside $Y_1$ for $x,y$ with action $\ge -D_1$ when $R\ll 0$ by the same neck-stretching argument along $Y_1$, we may assume  $NS_{2,0}(NS_{1,\frac{\epsilon_1}{2}}(J^1))\in \cJ^{\le k}_{reg,SFT,\le 1}(H_0)$. Therefore we can perturb $NS_{1,\frac{\epsilon_1}{2}}(J^1)\in \cJ^{\frac{\epsilon_1}{2},\le k}_{+,\le 1}(W_*)$ outside $W^1_*$ near orbits in $W^2_*$ to obtain $J^{2}_*$ such that $NS_{2,0}(J^2)\in \cJ^{\le k}_{reg,SFT,\le 2}(H_0)$, this will not influence the previous regular property for periodic orbits with action down to $-D_1$ by the integrated maximum principle. Then we can apply neck-stretching to $J^2_*$ at $Y_2$ to obtain $\epsilon_2$ with desired properties and keep the induction going. 
		
		Since we require that $H_R$ outside $Y_i$ is the same for $W_1,W_2$ whenever $R\le\frac{1}{i}$ and $\delta_i\le \frac{1}{i}$, it is clear that $\cM_{x,y,H_{R'R\delta_i}}$ using $NS_{i+1,R'}(NS_{i,R}(J^i_*))$ can be identified  for $R<\epsilon_i$ whenever the dimension $\le k$ and action of $x,y$ is greater than $-D_i$, as it is contained outside $Y_i$ where all the geometric data are the same.
\end{proof}

\begin{proposition}\label{prop:ind}
	Let $Y$ be a $k$-ADC contact manifold with two topologically simple Liouville fillings $W_1$ and $W_2$. Then for $*=1,2$, there exist a sequence of almost complex structures $\widetilde{J}^1_*,\widetilde{J}^2_*,\ldots$ and a sequence of positive numbers $1>R_1>R_2>\ldots >0$, such that for any oriented $k$-sphere bundles $E_*$ over $W_*$  with $E_1|_Y = E_2|_Y$, we have an isomorphism $\Phi:  \varinjlim_i H(\cC_+^{R_i,\widetilde{J}_1^i}) \simeq SH^*_+(W_1) \to SH^*_+(W_2)\simeq \varinjlim_i H(\cC_+^{R_i,\widetilde{J}_2^i})$ so that the following Euler part of the Gysin exact sequence commutes,
	$$
	\xymatrix{
	 \varinjlim_i H^{*}(\cC_{+,\le i}^{R_i,\widetilde{J}_1^i}) \ar[r]\ar[d]^{\Phi} & \varinjlim_i H^{*+k+1}(\cC_{+,\le i}^{R_i,\widetilde{J}_1^i})\ar[d]^{\Phi} \\
	 \varinjlim_i H^{*}(\cC_{+,\le i}^{R_i,\widetilde{J}_2^i}) \ar[r] & \varinjlim_i H^{*+k+1}(\cC_{+,\le i}^{R_i,\widetilde{J}_2^i})}
	$$
\end{proposition}
\begin{proof}
	Using the almost complex structures from Proposition \ref{prop:neck}, we define $\widetilde{J}^i_*$ to be $NS_{i,\frac{\epsilon_i}{2}}(J^i_*)$ for $*=1,2$. By Proposition \ref{prop:neck}, we have $\widetilde{J}^{i}_*\in \cJ^{\frac{\epsilon_i\delta_i}{2},\le k}_{+,\le i}(W_*)=\cJ^{\delta_{i+1},\le k}_{+,\le i}(W_*)$. Therefore Proposition \ref{prop:natural1}, the direct limit of following commutative sequence computes the Euler part of the Gysin exact sequence,
	$$\xymatrix{
	H^*(\cC^{\delta_2,\widetilde{J}^1_*}_{+,\le 1}) \ar[r]\ar[d] & H^*(\cC^{\delta_3,\widetilde{J}^2_*}_{+,\le 2}) \ar[r]\ar[d] &\ldots \\
	H^{*+k+1}(\cC^{\delta_2,\widetilde{J}^1_*}_{+,\le 1}) \ar[r] & H^{*+k+1}(\cC^{\delta_3,\widetilde{J}^2_*}_{+,\le 2}) \ar[r] & \ldots}
	$$
	
	We first show that the continuation map $H^*(\cC^{\delta_{i+1},\widetilde{J}^i_*}_{+,\le i}) \to  H^*(\cC^{\delta_{i+2},\widetilde{J}^{i+1}_*}_{+,\le i+1})$ is naturally identified for $*=1,2$. Note that the continuation map is decomposed into continuation maps $\Xi:H^*(\cC^{\delta_{i+1},\widetilde{J}^i_*}_{+,\le i})\to H^*(\cC^{\delta_{i+2}, NS_{i+1,\frac{\epsilon_{i+1}}{2}}(\widetilde{J}^{i}_*)}_{+,\le i})$ and $\Psi:H^*(\cC^{\delta_{i+2}, NS_{i+1,\frac{\epsilon_{i+1}}{2}}(\widetilde{J}^i_*)}_{+,\le i}) \to H^*(\cC^{\delta_{i+2},\widetilde{J}^{i+1}_*}_{+,\le i+1})$. Then $\Xi$ is identity by Proposition \ref{prop:natural2} using the regular homotopy $NS_{i+1,s}(\widetilde{J}^i_*)$  for $s\in [\frac{\epsilon_{i+1}}{2}, 1]$. Since $J^{i+1}_*$ is the same as $\widetilde{J}^i_*$ inside $W^i$, we have $NS_{i+1,\frac{\epsilon_{i+1}}{2}}(\widetilde{J}^i_*)$ is $\widetilde{J}^{i+1}_*$ inside $W^i$. Then the integrated maximum principle implies that $\Psi$ is the composition $$H^*(\cC^{\delta_{i+2}, NS_{i+1,\frac{\epsilon_{i+1}}{2}}(\widetilde{J}^i_*)}_{+,\le i}) \stackrel{=}{\to} H^*(\cC^{\delta_{i+2}, \widetilde{J}^{i+1}_*}_{+,\le i}) \stackrel{\subset}{\to}  H^*(\cC^{\delta_{i+2}, \widetilde{J}^{i+1}_*}_{+,\le i+1}),$$
 which is the same for $*=1,2$. Therefore all the horizontal arrows in the diagram can be identified for both $W_1$ and $W_2$. We still need to identify the vertical arrow, i.e.\ the Euler part of Gysin sequence. For $\cC^{\delta_{i+1},\widetilde{J}^i_*}_{+,\le i}$, we pick the parallel transport outside $Y_i$ such that they are identified for $*=1,2$, which is possible since $E|_{\partial W_1}=E|_{\partial W_2}$. Since the Euler part only requires $\cM_{x,y}$ with $\dim \cM_{x,y}\le k$ and parallel transport over them, \eqref{ns1} of Proposition \ref{prop:neck} implies that whole diagram can be identified for $*=1,2$. Then Proposition \ref{prop:natural1} implies the proposition. 
\end{proof}
\begin{remark}
	Using the almost complex structure satisfies the condition here and is close to the condition in \cite[Theorem A]{zhou2019symplectic}, we have the isomorphism in Proposition \ref{prop:ind} also yields the identification of the map $\delta_{\partial}:SH^\bullet_+(W_*) \to H^{\bullet+1}(Y)$ for $*=1,2$.
\end{remark}

\section{Proof of the main theorem and applications}\label{s4}
Our method of proving Theorem \ref{thm:main} is to represent even degree cohomology classes as Euler classes of sphere bundles. The following result explains which class can be realized as the Euler class of a sphere bundle.
\begin{theorem}[{\cite[Theorem 4.1]{guijarro2002bundles}}]\label{thm:euler}
	Given $k,m \in \N$, let $K(\Z,2k)^m$ be the $m$-skeleton of the  Eilenberg-MacLane space $K(\Z,2k)$, with inclusion map $i:K(\Z,2k)^m \hookrightarrow K(\Z,2k)$. Then there is an integer $N(k,m) > 0$  and an oriented $2k$-dimensional vector bundle $\xi_{k,m}$ over $K(\Z,2k)^m$ with $e(\xi_{k,m}) = N(k,m)\cdot i^*u$, where $u$ is the generator of $H^{2k}(K(\Z,2k);\Z)$. 
\end{theorem}
As a corollary of Theorem \ref{thm:euler}, let $W$ be a manifold of dimension $2n$ and $\alpha \in H^{2k}(W;\Z)$. Then $\alpha$ is uniquely represented by the homotopy class of a map $f_\alpha:M \to K(\Z,2k)^{2n+1}$. Then the Euler class of $f_\alpha^*\xi_{k,2n+1}$ is $N(k,2n+1)\cdot\alpha$. We first obtain the following proposition, which may have some independent interest.
\begin{proposition}\label{prop:lowdg}
	Let $Y$ be a $k$-ADC manifold with a topologically simple Liouville filling $W_1$ such that  $SH^*(W_1)=0$. Then for any other topologically simple Liouville filling $W_2$, we have $H^{2m}(W_2) \to H^{2m}(Y)$ is injective for $2m \le k+1$.
\end{proposition}
\begin{proof}
	Note that the abelian group $H^{2m}(W_2;\Z)$ has a non-conical decomposition into free and torsion parts $H^{2m}(W_2;\Z)=H^{2m}_{free}(W_2;\Z)\oplus H^{2m}_{tor}(W_2;\Z)$. To prove the injectivity of $H^{2m}(W_2) \to H^{2m}(Y)$ for real cohomology, it suffices to show that for any decomposition we have $H^{2m}_{free}(W_2;\Z) \to H^{2m}(Y;\Z)$ is an injection for $2m \le k+1$. Assume otherwise, there is an element $\alpha \in H^{2m}_{free}(W_2;\Z)\subset H^{2m}(W_2;\Z)$ such that $\alpha|_Y = 0$ in $H^{2m}(Y;\Z)$. By Theorem \ref{thm:euler}, there exist $N\in \N$, a $2m$-dimensional vector bundle $\xi_{m,2n+1}$ over the $2n+1$ skeleton $K(\Z,2m)^{2n+1}$ of the Eilenberg-MacLane space $K(\Z,2m)$ and $f_\alpha: W_2 \to K(\Z,2m)^{2n+1}$, such that the Euler class of $E_2:=f_\alpha^*\xi_{m,2n+1}$ is $N\alpha$. Since $\alpha|_Y = 0$ in $H^{2m}(Y;\Z)$, we have $f_\alpha|_Y: Y \to K(\Z,2m)^{2n+1}$ is contractible. Hence $E_2|_Y$ is a trivial bundle. Let $E_1\to W_1$ be the trivial sphere bundle. Therefore by Proposition \ref{prop:ind}, there exists 
	almost complex structures $J_*^1,J_*^2\ldots$ and $1>R_1>R_2>\ldots > 0$, such the Euler part for the positive symplectic cohomology for $E_1\to W_1$ and $E_2\to W_2$ can be identified. By \cite[Corollary B]{zhou2019symplectic}, we have $SH^*(W_2)=0$. Then Proposition \ref{prop:final} implies the following commutative diagram
	$$
	\xymatrix{H^{*+1}(W_1)\ar[rr]^{0} && H^{*+k+2}(W_1)\\
		\varinjlim H^*(\cC^{R_i,J^i_1}_{+,\le i}) \ar[rr]\ar[d]^{\simeq}\ar[u]_{\simeq} && H^{*+k+1}(\cC^{R_i,J^i_1}_{+,\le i})\ar[d]^{\simeq}\ar[u]_{\simeq}\\
		\varinjlim H^*(\cC^{R_i,J^i_2}_{+,\le i}) \ar[rr]\ar[d]^{\simeq} && H^{*+k+1}(\cC^{R_i,J^i_2}_{+,\le i})\ar[d]^{\simeq}\\
		H^{*+1}(W_2)\ar[rr]^{\wedge(-N\alpha) } && H^{*+k+2}(W_2)
	}
	$$
	We arrive at a contradiction, since $N\alpha \ne 0$.
\end{proof}	
Proposition \ref{prop:lowdg} says that if we have extra room in the positivity of the SFT degree and also the vanishing of symplectic cohomology, then $H^{*}(W)\to H^*(Y)$ is necessarily injective for low even degrees. For example, $Y^{2n-1}$ is flexibly fillable contact manifold, then $Y$ is $(n-3)$-ADC \cite{lazarev2016contact}. In this case, we have $H^*(W) \to H^*(Y)$ is always injective for even degree $*\le n-2$. Note that such property also follows from \cite[Corollary B]{zhou2019symplectic} that $H^*(W)\to H^*(Y)$ is independent of fillings and for $*<n-2$ we have $H^*(W)\to H^*(Y)$ is an isomorphism for Weinstein fillings. However, Proposition \ref{prop:lowdg} holds for very different reasons. Note that we do not assume $H^{2m}(W_1)\to H^{2m}(Y)$ is injective in Proposition \ref{prop:lowdg}.

\begin{proof}[Proof of Theorem \ref{thm:main}]
	By Proposition \ref{prop:ind}, we can pick $J^1_*,J^2_*,\ldots$ and $1>R_1>R_2>\ldots >0$ such that the Euler part of the positive symplectic cohomology for $W_1,W_2$ can be identified as long as $E_1|_Y=E_2|_Y$. Since $SH^*(W_1)=SH^*(W_2)=0$, we can define $\phi$ to be the composition of 
	$$H^{*}(W_1)\stackrel{\simeq}{\to} \varinjlim_i H^{*-1}(\cC^{R_i,J^i_1}_{+,\le i})\stackrel{\simeq}{\to} \varinjlim_i H^{*-1}(\cC^{R_i,J^i_2}_{+,\le i}) \stackrel{\simeq}{\to} H^{*}(W_2).$$
	In other words, $\phi$ is the identification in \cite[Corollary B]{zhou2019symplectic}, such that 
	\begin{equation}\label{eqn:com}
	\xymatrix{
	H^*(W_1)\ar[rr]^{\phi} \ar[rd] & & H^*(W_2) \ar[ld]\\
	& H^*(Y) &
    }
	\end{equation}
	is commutative. In particular, we have $\phi(1)=1$\footnote{Without \cite[Corollary B]{zhou2019symplectic}, $\phi(1)=\pm 1$ can already be obtained by grading.} and  and $\phi$ is actually induced from an isomorphism $\phi_{\Z}$ for $\Z$-coefficient cohomology.  We pick an element $\alpha_1\ne 0 \in H^{2k}(W_1;\Z)$ for $2k\le n-2$. By  \cite[Corollary B]{zhou2019symplectic},  let $\alpha_2=\phi_{\Z}(\alpha_1)\in H^{2k}(W_2;\Z)$,  then we have $\alpha_2|_Y=\alpha_1|_Y\in H^*(Y;\Z)$ by the $\Z$-coefficient version of \eqref{eqn:com}, i.e.\ \cite[Corollary B]{zhou2019symplectic}. By Theorem \ref{thm:euler}, there exist $N\in \N$ and a bundle $\xi_{k,2n+1}$, such that $E_*:=f_{\alpha_*}^*\xi_{k,2n+1}$ is a vector bundle over $W_*$ with Euler class $N\alpha_*$ for $*=1,2$, where $f_{\alpha_*}:W_*\to K(\Z,2k)^{2n+1}$ is the map representing $\alpha_*$.  Since $\alpha_2|_Y=\alpha_1|_Y\in H^*(Y;\Z)$, we have $f_{\alpha_1}|_Y$ is homotopic to $f_{\alpha_2}|_Y$. As a consequence, we have  $E_1|_Y = E_2|_{Y}$ and $e(E_1) = N\alpha_1,e(E_2)=N\alpha_2$. Then by the same argument in Proposition \ref{prop:lowdg}, we have the following commutative diagram
		$$
	\xymatrix{H^{*+1}(W_1)\ar[rr]^{\wedge(- N\alpha_1)} \ar@/_4.0pc/[ddd]^{\phi} && H^{*+k+2}(W_1) \ar@/^4.0pc/[ddd]^{\phi}\\
		\varinjlim H^*(\cC^{R_i,J^i_1}_{+,\le i}) \ar[rr]\ar[d]^{\simeq}\ar[u]_{\simeq} && H^{*+k+1}(\cC^{R_i,J^i_1}_{+,\le i})\ar[d]^{\simeq}\ar[u]_{\simeq}\\
		\varinjlim H^*(\cC^{R_i,J^i_2}_{+,\le i}) \ar[rr]\ar[d]^{\simeq} && H^{*+k+1}(\cC^{R_i,J^i_2}_{+,\le i})\ar[d]^{\simeq}\\
		H^{*+1}(W_2)\ar[rr]^{\wedge(-N\alpha_2) } && H^{*+k+2}(W_2)
	}
	$$
	Then we have  $\phi(N\alpha_1\wedge \beta)= N\alpha_2\wedge \phi(\beta)$. Since $\phi(1)=1$, it follows that $\phi(N\alpha_1)=N\alpha_2$ and $\phi(N\alpha_1\wedge \beta)= N\alpha_2\wedge \phi(\beta)=\phi(N\alpha_1)\wedge \phi(\beta)$. This finishes the proof.
\end{proof}

\begin{remark}
	Combining with the argument in this paper with \cite{zhou2019symplectic}, one can prove the following commutative diagram for a $k$ sphere bundle $E$ is independent of fillings and extensions of $E|_Y$ to $W$ as long as $Y$ is $k$-ADC
	$$\xymatrix{
	    SH^*_+(W) \ar[r]^{\delta_{\partial}} \ar[d]^{e(E)} & H^{*+1}(Y) \ar[d]^{e(E)} \\
	    SH^{*+k+1}_+(W) \ar[r]^{\delta_{\partial}}  & H^{*+k+2}(Y)  
       }$$
   As a corollary $\Ima \delta_{\partial}$ is closed under multiplication by the Euler class $e(E|_Y)$.  Then the argument of Theorem \ref{thm:main} implies that $\Ima \delta_{\partial}$ is closed under the multiplication by even degree $\le k+1$ elements in $\Ima \delta_{\partial}$. Note that $\Ima \delta_{\partial}$ is an interesting invariant of ADC manifolds and can be used to define obstructions to Weinstein fillability.
\end{remark}

Theorem \ref{thm:main} can be applied to examples listed in Example \ref{ex:ex}, the major class would be flexibly fillable contact manifolds. In the following, we list several cases, where the whole real cohomology ring is unique. For simplicity, we only consider simply connected contact manifolds. Note that the following Corollary includes Corollary \ref{cor:ring}.
\begin{corollary}\label{cor:ringlist}
	Let $Y$ be a simply connected flexibly fillable contact manifold satisfying one the following conditions, 
	\begin{enumerate}
		\item $Y$ is $4n+1$ dimensional for $n\ge 1$;
		\item\label{c2} $Y$ is $4n+3$ dimensional for $n\ge 1$, and the flexible filling $W$ has the property that for every $\alpha \wedge \beta \in H^{2n+2}(W)$ with $\deg(\alpha),\deg(\beta)$ odd, then $\alpha$ or $\beta$ can be decomposed into a nontrivial product;
	\end{enumerate}
then $H^*(W)$ as a ring is unique for any Liouville filling $W$ with $c_1(W)=0$.
\end{corollary}
\begin{proof}
 By Theorem \ref{thm:main} and \cite[Corollary B]{zhou2019symplectic}, the ring structure on $H^*(W)$ is unique if one of the multipliers is of even degree $\le \frac{\dim W}{2}-2$ or if the degree of product is $\le \frac{\dim W}{2}-1$. When $\dim Y=4n+1$, if the degree of the product is in the undetermined region, i.e.\ $\frac{\dim W}{2}=2n+1$, then one of the multipliers must be of even degree. If $Y$ is simply connected, then $H^1(W)$ is zero by \cite[Corollary B]{zhou2019symplectic}. As a consequence, the other odd degree multiplier must have degree $\ge 3$. Therefore the even degree multiplier has degree at most $\frac{\dim W}{2}-3$. In particular,  all products fall in the above two cases. Therefore the ring structure is unique. In case \eqref{c2}, the undetermined case is when the product has degree $2n+2$. If the product is from two classes of even degree, then we can apply Theorem \ref{thm:main}. If the product is from two classes of odd degree, then by assumptions that one of them can be reduced to a nontrivial product. Since the ring structure in that degree is unique, hence the decomposition exists for any other filling. Therefore the product can be rewritten as a product of two even elements. Hence the ring structure is unique.
\end{proof}

\begin{proof}[Proof of Corollary \ref{cor:real}]
By Corollary \ref{cor:ringlist}, the real cohomology ring of the filling is unique. Moreover, it is straightforward to verify that the cohomology ring of products of $\CP^n,\mathbb{HP}^n, S^{2n}$ and at most one copy of $S^{2n+1}$ for $n\ge 1$ has unique minimal models. By \cite[Theorem E]{zhou2019symplectic}, any exact filling of $\partial(\mathrm{Flex}(T^*M))$ with vanishing first Chern class is necessary simply connected, in which case, the real homotopy type is determined by the minimal model by \cite[\S 19]{bott2013differential}.
\end{proof}

Theorem \ref{thm:main} can only be applied when there is one filling with vanishing symplectic cohomology. In some cases, symplectic cohomology vanishes with non-trivial local systems \cite{albers2017local}. Here we only give one special example in such case. 

\begin{proposition}
	Assume $W$ is a Liouville filling of $Y:=\partial{T^*\CP^n}$ for $n\ge 3$ odd, which is $(2n-4)$-ADC. If $c_1(W)=0$ and $H^2(W;\Z)\to H^2(Y;\Z)$ is not zero. Then real cohomology ring $H^*(W)$ is isomorphic to $H^*(T^*\CP^n)$.
\end{proposition}
\begin{proof}
	Since $\CP^n$ is spin for $n$ odd, by \cite[Theorem D]{zhou2019symplectic}, there is a local system on both $W$ and $T^*\CP^n$, such that they are the same on $Y$ and the twisted symplectic cohomology vanishes for both $W$ and $T\CP^n$. Then we can apply the same argument of Theorem \ref{thm:main} to the case with local systems to finish the proof.
\end{proof}

\bibliographystyle{plain} 
\bibliography{ref}
\end{document}